\documentclass[11pt,a4paper,twoside]{article}
\usepackage{color}
\usepackage{bm, amsmath, amssymb, amsthm, graphicx} 

\newcommand\tr{\operatorname{trace}}

\def\eq{\hspace*{-2.5mm}&=&\hspace*{-2.5mm}}
\def\<{\langle}
\def\>{\rangle}

\newtheorem{corollary}{Corollary}

\newtheorem{example}{Example}
\newtheorem{remark}{Remark}
\newtheorem{lemma}{Lemma}

\newtheorem{theorem}{Theorem}

\author{Vladimir Rovenski \footnote{Department of Mathematics, University of Haifa, Mount Carmel, 3498838 Haifa,  Israel
       \newline e-mail: {\tt vrovenski@univ.haifa.ac.il}
       } 
}

\title{Willmore-type variational problem for foliated hypersurfaces}

\begin{document}

\date{}

\maketitle


\begin{abstract}
We study new Willmore-type variational problem for a hypersurface $M$ in $\mathbb{R}^{n+1}$ equipped with an $s$-dimensional foliation~${\cal F}$.
Its general version is the Reilly-type functional $WF_{n,s}=\int_M F(\sigma^{\cal F}_1,\ldots,\sigma^{\cal F}_s)\,{\rm d}V$,
where $\sigma^{\cal F}_i$ are elementary symmetric functions of the eigenvalues of the second fundamental form restricted on the leaves of $\cal F$.
The~first and second variations of such functionals are calculated, conformal invariance of some of $WF_{n,s}$ is also shown.
The Euler-Lagrange equation for a critical hypersurface with a transversally harmonic
(e.g., Riemannian) foliation $\cal F$ is found and examples with $s\le2$ and $s=n$ are considered.
Critical hypersurfaces of revolution are found, and it is shown that they are a local minimum for special~variations.


\end{abstract}

\section{Introduction}

Many authors, e.g., \cite{chen2,chen3,GTT-2019,MN-2014,Will,ZLW-2014},
were looking for an immersion $\phi: M^n\to\bar M^{n+1}$ of a smooth manifold $M^n\ (n\ge2)$ into a Riemannian manifold $(\bar M,\bar g)$, in particular, Euclidean space $\mathbb{R}^{n+1}$, which is a critical point of the following functionals for compactly supported variations of $\phi$:
\begin{equation}\label{E-func-0}
 W_{n,p}=\int_M H^{p}\,{\rm d}V ,\qquad
 J_{n,p}=\int_M \|{h}\,\|^{p}\,{\rm d}V.
\end{equation}
Here,
${h}$ is the scalar second fundamental form of $\phi(M)$, ${H}=\frac1n\tr_{g}{h}$ is the mean curvature,
${\rm d}V$ is the volume form of the induced metric $g$ on~$M$, and $p>0$.
These functionals
measure how much $\phi(M)$ differs from a minimal hypersurface ($H=0$) or a totally geodesic hypersurface ($h=0$).
The~actions \eqref{E-func-0} are a particular case of functionals
$WF_{n}=\int_M F(H)\,{\rm d}V$ and $JF_{n}=\int_M F(\|h\|)\,{\rm d}V$,
where $F$ is a $C^3$-regular function of one variable, e.g., \cite{Chang-13,Guo-07,Guo-09,LG-2021}.
For a closed oriented smooth hypersurface $M^n$ in ${\mathbb R}^{n+1}$, we~get $W_{n,n}\ge C_n$,
where $C_n=\frac{2\pi^{(n+1)/2}}{\Gamma((n+1)/2)}$ is the area of the unit $n$-sphere;
the equality $W_{n,n}=C_n$ holds if and only if $M^n$ is embedded as a hypersphere, see \cite{chen2}.

\smallskip

Variational problems for \eqref{E-func-0} were first posed by Thomas Willmore in \cite{Will} for $W_{2,2}$, which belongs to conformal geometry. The~Euler-Lagrange equation for $W_{2,2}$
is the well known elliptic PDE
\begin{align}\label{Eq-Will-2}
 \Delta H + 2\,H(H^2-K)=0,
\end{align}
where $\Delta$ is the Laplacian
and $K$ --
the gaussian curvature
of $M^2$.
Solutions of \eqref{Eq-Will-2} are called Willmore surfaces.
An~important class of Willmore surfaces in $\mathbb{R}^3$ arise as the stereographic projection of minimal surfaces
in the 3-sphere.
By Lawson's theorem, any compact orientable surface can be minimally embedded in the 3-sphere.
For a closed orientable surface $M$ in $\mathbb{R}^3$, the inequality $W_{2,2}\ge C_2=4\pi$ holds with the equality for round spheres.
If $M^2$ is a torus in $\mathbb{R}^3$, then, according the Willmore conjecture proven
by F.C.~Marques and A.~Neves in~\cite{MN-2014},
we have $W_{2,2}\ge 2\,\pi^2$; the equality holds if and only if the generating curve is a circle and the ratio of radii is $\frac1{\sqrt2}$.
Willmore surfaces have applications in biophysics,
materials science, architecture, etc., e.g.,~\cite{Song-2022}.

R. Reilly \cite{Reilly-73} and some mathematicians studied variations of more general functionals
than \eqref{E-func-0}:
\begin{align}\label{E-Fsigma}
 WF_{n}=\int_M F(\sigma_1,\ldots,\sigma_n)\,{\rm d}V ,\qquad
 JF_{n}=\int_M F(\tau_1,\ldots,\tau_n)\,{\rm d}V ,
\end{align}
where $F\in C^3(\mathbb{R}^n)$.
The~elementary symmetric functions $\sigma_r=\sum_{i_1<\ldots<i_r} k_{i_1}\ldots k_{i_r}\ (0\le r\le n)$ of the principal curvatures $k_i$ satisfy the
equality $\sum\nolimits_{\,r=0}^{\,n}\sigma_r t^r=\det({\rm id}_{\,TM}+tA)$, where $A$ is the Weingarten operator, i.e., $g(AX,Y)=h(X,Y)$.
The power sums of the principal curvatures, $\tau_i=k_1^i+\ldots+k_n^i={\rm trace}\, A^i$, can be expressed as polynomials of $\sigma_r$
using the Newton formulas, e.g., \cite{R2021}.
For example, $\sigma_0=1$, $\tau_1=\sigma_1=n H$, $\sigma_n=\det A$, $\tau_2=\|h\|^2$ and $2\,\sigma_2=\tau_1^2-\tau_2$.
The $r$-th ($r\le n$) order Willmore functional, introduced by Z.~Guo in~\cite{Guo-07},
\begin{align}\label{E-F-conf}
 W^{\rm conf}_{n,r} =
 \int_M Q_{\,r}^{\,n/r}\,{\rm d}V,\quad Q_r=\sum\nolimits_{j=0}^r (-1)^{j+1} C^j_r S_1^{r-j} S_j ,
\end{align}
is a special case of \eqref{E-Fsigma}, invariant under conformal group of $(\bar M, \bar g)$ and vanishing on totally umbilical hypersurfaces.
Here, $S_r=\sigma_r/C^r_n$ (where $C^r_n=\frac{n!}{r! (n-r)!}$ is a~binomial coefficient) is the $r$-th mean curvature function of a hypersurface.
In particular, $Q_2=S_1^2-S_2=\frac1{n^2(n-1)}\,((n-1)\sigma_1^2-2\,n\,\sigma_2)$.
Examples of hypersurfaces in $\mathbb{R}^{n+1}$ that are critical for \eqref{E-F-conf} are given in~\cite{Guo-09,LG-2021}.

\smallskip

An interesting problem is the generalization of the Willmore functional to submanifolds with additional structures, such as foliations or almost products. 
Let
$M^n\ (n\ge2)$ equipped with an $s$-dimensional $(1\le s\le n)$ foliation ${\cal F}$
be immersed into a Riemannian manifold $(\bar M,\bar g)$.
Let $h_{\cal F}$ be the restriction of the second fundamental form of $M$ on the leaves of~${\cal F}$.
Denote by $\tau^{\cal F}_i\ (1\le i\le s)$ the power sums, $\sigma^{\cal F}_r\ (1\le r\le s)$ elementary symmetric functions of the eigenvalues
$k^{\cal F}_1\le\ldots\le k^{\cal F}_s$
of $h_{\cal F}$, and set $S^{\cal F}_r=\sigma^{\cal F}_r/C^r_s$.
We have
 $\tau^{\cal F}_1=\sigma^{\cal F}_1=s H_{\cal F}=\tr_g h_{\cal F}$, $\tau^{\cal F}_2=\|h_{\cal F}\|^2$,
$(\tau^{\cal F}_1)^2 -\tau^{\cal F}_2 = 2\,\sigma^{\cal F}_2$, etc.
For~foliation theory we refer to \cite{CC-2000}, the extrinsic geometry of foliations was developed in \cite{R2021}.
We~study Reilly-type functionals for compactly supported variations of $(M^n, {\cal F})$ immersed in $\mathbb{R}^{n+1}$:
\begin{align}\label{E-FFsigma}
 WF_{n,s}=\int_M F(\sigma^{\cal F}_1,\ldots,\sigma^{\cal F}_s)\,{\rm d}V ,\qquad
 JF_{n,s}=\int_M F(\tau^{\cal F}_1,\ldots,\tau^{\cal F}_s)\,{\rm d}V ,
\end{align}
which for~$s=n$ reduce to~\eqref{E-Fsigma}.
%
For $F=(\sigma_1^{\cal F}/s)^p$ and $F=(\tau_2^{\cal F})^{p/2}$,
the actions \eqref{E-FFsigma} read~as
\begin{equation}\label{E-func}
 W_{n,p,s} = \int_M {H}_{\cal F}^{p}\,{\rm d}V,\qquad
 J_{n,p,s}=\int_M \|h_{\cal F}\|^{p}\,{\rm d}V ,
\end{equation}
which reduce to~\eqref{E-func-0} for~$s=n$.

\begin{remark}\rm
A foliated hypersurface in $\mathbb{R}^{n+1}$, whose leaves $\{L\}$ are minimal submanifolds in $\mathbb{R}^{n+1}$
is an example of a minimizer for \eqref{E-func}$_1$ with even $p$.
A foliated hypersurface in $\mathbb{R}^{n+1}$ with an asymptotic distribution $T{\cal F}$
(in particular, a ruled hypersurface) is a minimizer for \eqref{E-func}$_2$.
%
It is interesting to find critical hypersurfaces for actions \eqref{E-func} with ${H}_{\cal F}\ne0$ or $h_{\cal F}\ne0$ on an open dense set of~$M$.
\end{remark}

The~following special case of \eqref{E-FFsigma} is invariant under conformal group of $(\bar M, \bar g)$, see Theorem~\ref{Prop-conf}:
\begin{align}\label{E-Fs-conf}
 W^{\rm conf}_{n,s,r} =
 \int_M (Q^{\cal F}_r)^{n/r}\,{\rm d}V,\quad Q^{\cal F}_r =\sum\nolimits_{j=0}^r (-1)^{j+1} C^j_r (S^{\cal F}_1)^{r-j} S^{\cal F}_j ,
 \quad r\le s,
\end{align}
and reduces to \eqref{E-F-conf} for $s=n$.
Note that $Q^{\cal F}_2=(S^{\cal F}_1)^2-S^{\cal F}_2=\frac1{s^2(s-1)}\,((s-1)(\sigma^{\cal F}_1)^2-2s\,\sigma^{\cal F}_2)$
and $Q^{\cal F}_r=0$ if $k^{\cal F}_1=\ldots=k^{\cal F}_s$, e.g., for hypersurfaces of revolution in $\mathbb{R}^{n+1}$ foliated by~parallels.

We~hope that foliated hypersurfaces, which are local minima for \eqref{E-FFsigma}, will be useful for natural sciences and technology related to layered (laminated) non-isotropic~materials.

\smallskip

The paper is organized as follows.
Section~\ref{sec:prel} contains some lemmas (proved in Sect.~\ref{sec:07}) that help us calculate variations of Reilly-type functionals
on foliated hypersurfaces in $\mathbb{R}^{n+1}$.
In~Section~\ref{sec:02},
conformal invariance of \eqref{E-Fs-conf} is shown,
the first variations of the functionals \eqref{E-FFsigma}--\eqref{E-Fs-conf} are found, and
the corresponding Euler-Lagrange equations for the case of transversally harmonic (for example, Riemannian) foliation are obtained.
Then the second variations on critical hypersurfaces of some Willmore type functionals are calculated.
In Section~\ref{sec:04}, applications to hypersurfaces with low-dimensional foliations are given,
the critical hypersurfaces of revolution
for the actions \eqref{E-func} are presented,
and it is shown that they are a local minimum for special variations.

\section{Auxiliary results}
\label{sec:prel}

Let ${\bf r} : M^n\to\mathbb{R}^{n+1}$ be an immersion of a manifold $M$ into $\mathbb{R}^{n+1}$
with Euclidean metric $\bar g$ and the Levi-Civita connection $\bar\nabla$.
We identify $M$ with its image ${\bf r}(M)$ and restrict our calculations to a relatively compact neighborhood $U\subset M$ with
induced metric $g=\<\cdot\,,\cdot\>$ and normal coordinates $(x^1,\ldots,x^n)$ centered at a point $x\in M$.
Thus, $g_{ij}=\delta_{ij}$ (the Kronecker symbol) and $\Gamma^k_{ij}=0$ at
$x$.
%
Differentiation of a function $f$ (or a tensor) with respect to the variable $x^i$ will be denoted by $f_i$.

Let $\partial_i=\partial/\partial x^i$
be the coordinate vector fields on $U$.
So, the vectors ${\bf r}_i=\bar\nabla_{\partial_i}{\bf r}$ form a local coordinate basis for the tangent bundle $TM$ along $U$,
and we get $g=g_{ij}\,dx^i\,dx^j$, where $g_{ij}= \bar g({\bf r}_i,{\bf r}_j)=\<{\bf r}_i,{\bf r}_j\>$
and Einstein summation rule is used.
Let ${\bf N}$ be a unit normal vector field to $M$ on $U$.
The~vectors ${\bf N}_i=\bar\nabla_{\partial_i}{\bf N}$
belong to the tangent space $T_xM$, i.e.,
 $\<{\bf N}_i, {\bf N}\> =0$.

Let ${h}$ be the scalar second fundamental form of $M$ with respect to unit normal ${\bf N}$,
$A=-\bar\nabla\,{\bf N}$ the Weingarten operator
and ${H}=\frac1n\tr_g {h}$ the mean curvature.
Denote by $h^j$ the symmetric tensor dual to $A^j$, i.e., $h^j(X,Y)=\<A^j X,Y\>$.
For example, ${h}^2 = g^{kl} h_{li} h_{kj}\,dx^i dx^j = {h}^k_i h_{kj}\,dx^i dx^j$.

\smallskip

Consider a one-parameter family of hypersurfaces
 ${\bf r}_t = {\bf r} + t\,u\,{\bf N}\ (|t|<1)$.
We get a variation $\delta\,{\bf r} = u\,{\bf N}$, where $\delta=(d/dt)|_{\,t=0}$ is the variational derivative operator,
and $u:U\to\mathbb{R}$ is a smooth function supported on relatively compact neighborhood $U$. Obviously,
 $(\delta\,{\bf r})_i = u_i\,{\bf N} + u\,{\bf N}_i$.
The Hessian of a function $u$
is a (0,2)-tensor
 $({\rm Hess}_{\,u})(X,Y) = X(Y(u)) - (\nabla_X\,Y)u = (u_{ij}-\Gamma^k_{ij} u_k)dx^i dx^j$,
see \cite[pp.~48, 69]{petersen}. The~Laplacian is
 $\Delta u = \tr_g\,{\rm Hess}_{\,u}
 = g^{ij}u_{ij}$.
Note that $\< g, {\rm Hess}_{\,u} \> = \Delta\,u$. The~divergence of a vector field $X=X^i\partial_i$ on $M$ is
 ${\rm div}\,X = \nabla_i X^i$.

\begin{lemma}[see \cite{GTT-2019} for $n=2$]\label{Le-01} The following evolution equations are true:
\begin{align}
\label{Eq-01a}
 & \delta\,g_{ij} = - 2\,u\,h_{ij},
 \\
\label{Eq-01b}
 & \delta\,g^{ij} = 2\,u\,h^{ij},\\
\label{E-h-beta}
 & \delta\,h_{ij}
 = u_{ij} - u {h}^l_i\,h_{jl} \ \Leftrightarrow \
 \delta\,h = {\rm Hess}_{\,u} - u\,h^2, \\
\label{Eq-01h}
 & \delta\,\|h\|^2 =
 2\,\<h,\ u{h}^2 + {\rm Hess}_{\,u}\>,\\
\label{Eq-01e}
 & \delta (n H) = \Delta\,u + u\,\|h\|^2,\\
\label{Eq-01d}
 & \delta\,{\rm d}V = - n\,u H\,{\rm d}V.
\end{align}
\end{lemma}

We carry out further calculations for
a foliated hypersurface $(M,{\cal F})$
and a foliated neighborhood $U\subset M$ with normal coordinates $(x^1,\ldots,x^n)$ adapted to ${\cal F}$,
i.e., $(x^1,\ldots,x^s)$ are variables along the leaves, see \cite{CC-2000}.
Let $\nabla^{\cal F}:TM\times T{\cal F}\to T{\cal F}$ be the induced Levi-Civita connection on the subbundle $T{\cal F}\subset TM$ and on the leaves of ${\cal F}$.
The leafwise Laplacian on functions
is $\Delta_{\,\cal F} = \tr_g\,{\rm Hess}^{\cal F} = {\rm div}_{\cal F}\circ \nabla^{\,\cal F}$,
where ${\rm Hess}^{\cal F}$ is the Hessian on the leaves of $\cal F$.
Let $P:TM\to T{\cal F}$ be the orthoprojector, thus $P^2=P$ and $P$ is self-adjoint.
For $h_{\cal F}$ and its dual self-adjoint operator $A_{\cal F}$ we can write
 $h_{\cal F}(X,Y)= h(PX,PY)\ (X,Y\in\mathfrak{X}_M)$
 and
 $A_{\cal F}=P A P$.
Let $h_{\cal F}^j$ be the symmetric tensor dual to $A_{\cal F}^j$, i.e., $h_{\cal F}^j(X,Y)=\<A_{\cal F}^j X,Y\>$.
The~symmetric tensor
$h_{\rm mix}$
is given by
\[
 h_{\rm mix}(X,Y) = \frac12\,\big(h(PX,Y) + h(X,PY)\big) - h(PX,PY),\quad X,Y\in\mathfrak{X}_M.
\]
We have $\<h_{\cal F}, h_{\rm mix}\>=0$.
The equality $h_{\rm mix}=0$ means that $P A = A P$,
i.e., $T{\cal F}$ is an invariant subbundle for
$A$.
Let $h_{\cal F^\perp}$ be the restriction of $h$ on the normal distribution to ${\cal F}$ in $M$, then
$h_{\cal F^\perp}^2 = g^{\gamma\epsilon} h_{\alpha\gamma} h_{\alpha\epsilon}dx^\alpha dx^\beta$, where
$s<\alpha,\beta,\gamma,\epsilon\le n$.
Define symmetric~tensors
 $h_{\rm mix}^2 = g^{\alpha\beta} h_{\alpha i} h_{\beta j}dx^i dx^j + g^{ij} h_{\alpha i} h_{\beta j}\,dx^\alpha dx^\beta$
 and
 ${\rm Hess}^{\,\rm mix}_{\,u} = g^{ij} g^{\alpha\beta} u_{i\alpha}\,dx^j dx^\beta$,
where $1\le i,j\le s$ and $s<\alpha,\beta,\gamma,\epsilon\le n$.
Let~$A_{\,\rm mix}$ be the (1,1)-tensor dual to $h_{\rm mix}$, then $A^2_{\,\rm mix}$ is dual to~$h^2_{\,\rm mix}$.

The {Newton transformations} $T_{r}(A_{\cal F})$ of $A_{\cal F}$ are defined inductively or explicitly by, e.g. \cite{R2021},
\begin{eqnarray*}
  T_0(A_{\cal F}) \eq {\rm id}_{\,T\cal F},\quad T_r(A_{\cal F})=\sigma^{\cal F}_r\,{\rm id}_{\,T\cal F} -A_{\cal F}\,T_{r-1}(A_{\cal F})\quad (0< r\le s),\\
 T_r(A_{\cal F}) \eq \sum\nolimits_{j=0}^r(-1)^j\sigma^{\cal F}_{r-j}\,A_{\cal F}^j
 = \sigma^{\cal F}_r\,{\rm id}_{\,T\cal F} -\sigma^{\cal F}_{r-1}\,A_{\cal F} +\ldots +(-1)^r A_{\cal F}^{\,r}.
\end{eqnarray*}
For example, $T_1(A_{\cal F})=\sigma^{\cal F}_1\,{\rm id}_{\,T\cal F} - A_{\cal F}$ and $T_s(A_{\cal F})=0$ and the following equalities are true:
\begin{align}\label{E-NewtonT}
\nonumber
 & \tr T_r(A_{\cal F}) = (s-r)\,\sigma^{\cal F}_r,\quad \tr(A_{\cal F}\,T_r(A_{\cal F})) = (r+1)\,\sigma^{\cal F}_{r+1},\\
 & \tr(A_{\cal F}^2\,T_r(A_{\cal F})) = \sigma^{\cal F}_{1}\sigma^{\cal F}_{r+1} -(r+2)\,\sigma^{\cal F}_{r+2} .
\end{align}
The ``musical" isomorphism $\sharp:T^*M\to TM$
is used for tensors, e.g., $h^\sharp=A$, and for $(0,2)$-tensors $B$ and $C$ we have $\<B, C\> =\tr(B^\sharp C^\sharp)=\<B^\sharp, C^\sharp\>$.

\begin{lemma}\label{Prop-02} The variations of $\tau^{\cal F}_i$ and $\sigma^{\cal F}_r$ are the following:
\begin{align}\label{E-var-TS}
\nonumber
 & \frac1i\,\delta\,\tau^{\cal F}_i = \< h_{\cal F}^{i-1}, {\rm Hess}^{\cal F}_u\> + u\,(\tau^{\cal F}_{i+1} +\<h_{\cal F}^{i-1}, h^2_{\,\rm mix}\>),\\
 & \delta\,\sigma^{\cal F}_r = \<T_{r-1}(A_{\cal F}), {\rm Hess}^{\cal F\sharp}_u\> +u\big( \sigma^{\cal F}_1\sigma^{\cal F}_{r-1} -(r+1)\,\sigma^{\cal F}_{r+1}
 +\<T_{r-1}(A_{\cal F}),\, A_{\,\rm mix}^2\>\big).
\end{align}
\end{lemma}

\begin{proof}
By \eqref{E-h-beta}, we obtain
 $\delta A_{\cal F} = {\rm Hess}^{\cal F\sharp}_u + u(A_{\cal F}^2 + A_{\,\rm mix}^2)$.
Using this, \eqref{E-NewtonT} and the the following variations of $\tau^{\cal F}_i$ and $\sigma^{\cal F}_r$, see \cite[Sect.~1.5.3]{R2021}:
\begin{align*}
 \delta\,\tau^{\cal F}_i = i\tr(A_{\cal F}^{i-1}\delta A_{\cal F}),\quad
 \delta\,\sigma^{\cal F}_r = \tr(T_{r-1}(A_{\cal F})\,\delta A_{\cal F}),
\end{align*}
we get \eqref{E-var-TS}.
\end{proof}

\begin{lemma}\label{Lem-03} The following evolution equations are true:
\begin{align}
\label{Eq-02a}
 & \delta\,(s\,{H}_{\cal F}) = \Delta_{\,\cal F}\,u +u\,(\|h_{\cal F}\|^2 - \|h_{\rm mix}\|^2) ,\\
\label{Eq-02h}
 & \delta\,\|h_{\cal F}\|^2 =  2\,\<h_{\cal F}, \ u\,(h_{\cal F}^2 + h_{\rm mix}^2) + {\rm Hess}^{\,\cal F}_{\,u}\> , \\
 \label{E22}
 &\delta\,\|h_{\rm mix}\|^2 = u\,\<h_{\cal F}+h_{\cal F^\perp}, h_{\rm mix}^2\> + 2\,\<h_{\rm mix}, {\rm Hess}^{\,\rm mix}_{\,u}\> .
\end{align}
For any smooth function $f:M\times\mathbb{R}\to\mathbb{R}$ the following evolution is true:
\begin{align}\label{E21}
 &\delta\,(\Delta_{\,\cal F}f) = \Delta_{\,\cal F}\dot f + 2\,u\<h_{\cal F}, {\rm Hess}^{\,\cal F}_f\>
 {+} s u\<\nabla^{\cal F}H_{\cal F}, \nabla^{\cal F} f\> {+} 2\,h(\nabla^{\,\cal F} u, \nabla^{\cal F} f)
 {-} s H_{\cal F}\<\nabla^{\cal F} u, \nabla^{\cal F} f\> .
\end{align}
\end{lemma}

\begin{remark}\rm
The equations \eqref{Eq-02a}--\eqref{Eq-02h} can be deduced from Lemma~\ref{Prop-02}, but we will prove them in Section~\ref{sec:07}.
To find the second variation of the functionals \eqref{E-FFsigma} we need also the variation
$\delta\,\<h, {\rm Hess}^{\,\cal F}_{\,u}\>$, but we omit this calculation and consider the second variation of the functionals \eqref{E-func}~only.
\end{remark}


\noindent\ \
The following property helps to find the Euler-Lagrange equations using first variation of~\eqref{E-Fsigma}--\eqref{E-func}:
\begin{align}\label{E-divP-0}
 ({\rm div}\,P)\circ P=0 .
\end{align}
Here, $({\rm div}\, P) X = \sum\nolimits_{\,i=1}^n \<(\nabla_{e_i} P)X, e_i\>$, where $e_1,\ldots,e_n$ is a local orthonormal basis on~$M$.
Note that Riemannian foliations (and the leaves of twisted products, e.g., \cite{R2021}) satisfy~\eqref{E-divP-0}.

\begin{lemma}
A foliated Riemannian manifold $(M,g,{\cal F})$ satisfies
\eqref{E-divP-0}
if and only if ${\cal F}$ is transversally harmonic,
i.e., the normal distribution has zero mean curvature.
\end{lemma}

\begin{proof}
Using a local orthonormal frame on $M$ such that $e_i\in T{\cal F}\ (1\le i\le s)$, we calculate:
\begin{align*}
 &({\rm div}\,P)(PX) = \sum\nolimits_{\,i=1}^n \<(\nabla_{e_i} P)(PX), e_i\>
 = \sum\nolimits_{\,i=1}^n\big\{\<\nabla_{e_i} (P^2 X), e_i\> - \<P\nabla_{e_i} (P X), e_i\> \big\} \\
 & = \sum\nolimits_{\,i=1}^n\big\{\<\nabla_{e_i} (P X), e_i\> - \<\nabla_{e_i} (P X), P e_i\> \big\}
  = \sum\nolimits_{\,i>s} \<\nabla_{e_i} (P X), e_i\> = - \<X, (n-s)H^\perp\>,
\end{align*}
where
$(n-s)H^\perp = P\sum_{\,i>s} \nabla_{e_i}e_i$
is the mean curvature vector of
$(T{\cal F})^\perp$ and $X\in\mathfrak{X}_M$.
\end{proof}

For any 2-tensor $B$ on $M$ define the adjoint of the covariant derivative $\nabla^* B = -\sum_{\,i} (\nabla_i\,B)(e_i, \cdot)$, see \cite[p. 59]{petersen}.
We have the formula
 $\int_M \< B, \nabla B'\>\,{\rm d}V = \int_M  \< \nabla^* B, B'\>\,{\rm d}V$,
see \cite[p. 60]{petersen}; thus,
\begin{align}\label{EV-nabla}
 \int_M \< B, {\rm Hess}_{\,u} \>\,{\rm d}V
 = \int_M \< B, \nabla(\nabla u)\>\,{\rm d}V
 = \int_M \< \nabla^* B, \nabla u \>\,{\rm d}V
 = \int_M u\,(\nabla^*)^2(B)\,{\rm d}V.
\end{align}

The next lemma generalizes  \eqref{EV-nabla} and the well known Green's formula, e.g., \cite[p.~75]{petersen}.

\begin{lemma}\label{Lem-02}
If a foliated Riemannian manifold $(M,g,{\cal F})$ satisfies \eqref{E-divP-0}, then the following formulas are valid for any
compactly supported functions $u,f$ and 2-tensor $B$:
\begin{align}\label{E-divP-2}
 & \int_M f(\Delta_{\,\cal F}\,u)\,{\rm d}V = \int_M u(\Delta_{\,\cal F}\,f)\,{\rm d}V , \\
\label{EV2-nabla}
 & \int_M \< B, {\rm Hess}^{\cal F}_{\,u} \>\,{\rm d}V
 = \int_M u\,(\nabla^{{\cal F}*})^2(B)\,{\rm d}V.
\end{align}
\end{lemma}

\section{Main results}
\label{sec:02}

In Sect.\ref{sec:02a}, we find the Euler-Lagrange equations (and first variations) for the functionals \eqref{E-FFsigma}--\eqref{E-Fs-conf},
and in Sect.~\ref{sec:02b} we find the second variations of \eqref{E-FFsigma} and \eqref{E-func}.
First we check the conformity of~\eqref{E-Fs-conf}.

\begin{theorem}\label{Prop-conf}
The functional $W^{\rm conf}_{n,s,r}$ is a conformal invariant of a foliated hypersurface $(M,{\cal F})$ in a Riemannian manifold $(\bar M,\bar g)$.
\end{theorem}

\begin{proof}
Define a new Riemannian metric
on $\bar M$ by $\bar g^c=\mu^2\bar g$ for some positive function $\mu\in C^3(\bar M)$.
Then $g^c=\mu^2 g$ is the new induced metric on $M$,
thus, the new volume form of $M$ is ${\rm d}V^c=\mu^{n}\,{\rm d}V$.
If~$X$ is a $\bar g$-unit vector, then $X^c=X/\mu$ is a $g^c$-unit vector.
By the well known formula for the Levi-Civita connection, e.g., \cite[p.~14]{R2021}, we get
\[
 2\,\bar\nabla^{\,c}_X\,Y = 2\,\bar\nabla_X\,Y +\mu^{-2}\big( X(\mu^2)Y + Y(\mu^2)X - \<X,Y\>\bar\nabla \mu^2\big).
\]
By this with $X\in T{\cal F}$ and $Y={\bf N}^{\,c}$, the operators $A$ and $A^c$ are related~by
 $A^c=\frac{1}{\mu}\big(A-\frac{1}{\mu}\langle\bar\nabla\mu,N\rangle\,{\rm id}_{\,T\cal F}\big)$,
see also \cite[Sect.~2.1.2.2]{R2021}. By the above and $A_{\cal F}=P A P$, $A^{\,c}_{\cal F}=P A^{c} P$, we~get
\[
 A_{\cal F}^c=\frac{1}{\mu}\big(A_{\cal F}-\frac{1}{\mu}\langle\bar\nabla\mu,N\rangle\,{\rm id}_{\,T\cal F}\big),\quad
 H_{\cal F}^c = \frac1s\tr A_{\cal F}^c = \frac{1}{\mu}\big(H_{\cal F}-\frac{1}{\mu}\langle\bar\nabla\mu,N\rangle\big).
\]
Set $B_{\cal F}=H_{\cal F}\,{\rm id}_{\,T\cal F} - A_{\cal F}$.
Let $\lambda^B_i$ be the eigenvalues of $B_{\cal F}$ on ${\cal F}$ and $\sigma^B_r$ be elementary symmetric functions of~$B_{\cal F}$.
Obviously, $B_{\cal F}^{\,c}=\frac1\mu\,B_{\cal F}$ holds; hence, $\lambda^{B,c}_i=\frac1\mu\,\lambda^B_i$.
One can show that $Q^{\cal F}_r =-\sigma^B_r/C^s_r$ is true, see \cite[Lemma~2.1]{Guo-07}.
By the above, $\mu^r Q^{{\cal F},c}_r=Q^{\cal F}_r$ holds.
Hence, $(Q^{\cal F}_r)^{n/r} {\rm d}V$ is a conformal invariant of $(M,{\cal F})$ in $(\bar M,\bar g)$:
 $(Q^{{\cal F},c}_r)^{n/r} {\rm d}V^c = (Q^{\cal F}_r)^{n/r} {\rm d}V$.
Note that if $A_{\cal F}$ is a conformal operator on~$T{\cal F}$ (i.e., proportional to ${\rm id}_{\,T\cal F}$), then $B_{\cal F}=0$, hence, $Q^{\cal F}_r=0$.
\end{proof}

\subsection{The first variation}
\label{sec:02a}

We can state our main theorem.

\begin{theorem}\label{T-main1}
If \eqref{E-divP-0} is valid, then Euler-Lagrange equations for the functionals \eqref{E-FFsigma} are
\begin{align}\label{E-main1ab}
\nonumber
 &\sum\nolimits_{\,r=1}^s \big\{ (\nabla^{{\cal F}*})^2\big(F'_r\cdot T_{r-1}(A_{\cal F})\big)
 + F'_r\big(\sigma^{\cal F}_1\sigma^{\cal F}_{r-1} {-}(r+1)\sigma^{\cal F}_{r+1} {+}\<T_{r-1}(A_{\cal F}), A_{\,\rm mix}^2\>\big)\big\} - n F H  = 0 ,\\
 & \sum\nolimits_{\,i=1}^s \frac1i\,\big\{(\nabla^{{\cal F}*})^2\big(F'_i\cdot A_{\cal F}^{i-1}\big)
 + F'_i \big(\tau^{\cal F}_{i+1} +\<h_{\cal F}^{i-1}, h^2_{\,\rm mix}\> \big)\big\} - n F H = 0.
\end{align}
\end{theorem}

\begin{proof}
Using \eqref{Eq-01d}, we get the following:
\begin{align}\label{E-main2a}
\nonumber
 & \delta\,WF_{n,s}=\int_M \delta\,(F(\sigma^{\cal F}_1,\ldots,\sigma^{\cal F}_s)\,{\rm d}V)
  = \int_M \Big\{\sum\nolimits_{r=1}^s F'_r\cdot\delta\,\sigma^{\cal F}_r - n\,u F H\Big\}{\rm d}V,\\
 &\delta\,JF_{n,s}=\int_M \delta\,(F(\tau^{\cal F}_1,\ldots,\tau^{\cal F}_s)\,{\rm d}V)
  = \int_M \Big\{\sum\nolimits_{i=1}^s F'_i\cdot\delta\,\tau^{\cal F}_i - n\,u F H\Big\}{\rm d}V .
\end{align}
From \eqref{E-main2a} and \eqref{E-var-TS} we find the first variations of functionals \eqref{E-FFsigma}:
\begin{align}\label{E-main1a}
\nonumber
 & \delta\,WF_{n,s}=\int_M \Big\{\sum\nolimits_{\,r=1}^s\<F'_r\cdot T_{r-1}(A_{\cal F}), {\rm Hess}^{\cal F\sharp}_u\> \\
\nonumber
 &\quad + u\sum\nolimits_{\,r=1}^s F'_r\big(\sigma^{\cal F}_1\sigma^{\cal F}_{r-1} -(r+1)\sigma^{\cal F}_{r+1}
 +\<T_{r-1}(A_{\cal F}), A_{\,\rm mix}^2\>\big) - n\,u F H\Big\}{\rm d}V ,\\
 &\delta\,JF_{n,s} = \int_M \Big\{\sum\nolimits_{\,i=1}^s \frac1i\,F'_i
 \big(\< h_{\cal F}^{i-1}, {\rm Hess}^{\cal F}_u\> + u\,(\tau^{\cal F}_{i+1} +\<h_{\cal F}^{i-1}, h^2_{\,\rm mix}\>) \big)
 - n\,u F H\Big\}{\rm d}V .
\end{align}
From \eqref{E-main1a}, using \eqref{E-divP-0} and \eqref{E-divP-2}, we obtain \eqref{E-main1ab}.
\end{proof}

Equations \eqref{EVar-02} and \eqref{EVar1-02h} of the next statement follow from Theorem~\ref{T-main1}, but we will prove them.

\begin{corollary}
If \eqref{E-divP-0} is valid, then the Euler-Lagrange equations for the functionals $W_{n,p,s},\,J_{n,p,s}$, see \eqref{E-func},
and $W^{\rm conf}_{n,s,2}$, see \eqref{E-Fs-conf}, are, respectively, the following:
\begin{align}\label{EVar-02}
 & \Delta_{\,\cal F}\,({H}_{\cal F}^{p-1}) + {H}_{\cal F}^{p-1}\big(\|h_{\cal F}\|^2 -\|h_{\rm mix}\|^2 -\frac{n\,s}p\,H {H}_{\cal F}\big) = 0 ,\\
\label{EVar1-02h}
 & (\nabla^{{\cal F}*})^2(\|h_{\cal F}\|^{p-2} h_{\cal F} )
 + \|h_{\cal F}\|^{p-2}\big(\<h_{\cal F}, \ h_{\cal F}^2 + h_{\rm mix}^2\> - \frac np\,\|h_{\cal F}\|^2\,H\big) = 0 , \\
\label{EVar1-conf}
\nonumber
 & \Delta_{\,\cal F}\,\big((Q^{\cal F}_2)^{n/2-1}\sigma^{\cal F}_1\big)
 - \frac s{s-1}\,(\nabla^{\cal F*})^2\big((Q^{\cal F}_2)^{n/2-1} T_{1}(A_{\cal F}) \big)
 + \big\{\sigma^{\cal F}_1 (\sigma^{\cal F}_1 -2\,\sigma^{\cal F}_{2} +\|A_{\,\rm mix}\|^2) \\
 & \quad - \frac s{s-1}\,(\sigma^{\cal F}_1\sigma^{\cal F}_{2}  - 3\,\sigma^{\cal F}_{3} +\<T_{1}(A_{\cal F}),\, A_{\,\rm mix}^2\>)
 - s^2 Q^{\cal F}_2 H \big\}(Q^{\cal F}_2)^{n/2-1} =0.
\end{align}
\end{corollary}

\begin{proof}
a) Using \eqref{Eq-01d}, \eqref{Eq-02a} and \eqref{Eq-02h} we calculate the variation
\begin{align*}
 \delta\,({H}_{\cal F}^{p}\,{\rm d}V)
 & = {H}_{\cal F}^{p-1}\Big\{\frac ps\,\big(\Delta_{\,\cal F}\,u + u\,(\|h_{\cal F}\|^2 - \|h_{\rm mix}\|^2)\big)
 - n u {H} {H}_{\cal F}\Big\}{\rm d}V ,\\
 & \delta\,(\|h_{\cal F}\|^{p}\,{\rm d}V) = \|h_{\cal F}\|^{p-2}\big\{ p\,\<h_{\cal F}, \ u\,(h_{\cal F}^2 + h_{\rm mix}^2) + {\rm Hess}^{\,\cal F}_{\,u}\>
 - n\,u\,\|h_{\cal F}\|^2\,{H} \big\}{\rm d}V.
\end{align*}
Hence, using \eqref{E-divP-0}, \eqref{E-divP-2}, \eqref{EV2-nabla} and \eqref{Eq-01d}, we find the first variation of the actions~\eqref{E-func}:
\begin{align}\label{E-I1}
\nonumber
 \delta\,W_{n,p,s} & = \int_M \delta\,({H}_{\cal F}^{p}\,{\rm d}V)
 = \int_M  {H}_{\cal F}^{p-1}\Big\{\frac ps\,\big(\Delta_{\,\cal F}\,u
 + u\,(\|h_{\cal F}\|^2 - \|h_{\rm mix}\|^2)\big) - n\,u\,{H}_{\cal F} H \Big\}{\rm d}V \\
 & =\frac ps\int_M u \Big\{\Delta_{\,\cal F}({H}_{\cal F}^{p-1})
  + {H}_{\cal F}^{p-1}\,\big(\|h_{\cal F}\|^2 - \|h_{\rm mix}\|^2 - \frac{n\,s}p\,{H} {H}_{\cal F}\big) \Big\}{\rm d}V ,\\
\label{E-I2}
\nonumber
 \delta\,J_{n,p,s}  & = \int_M \|h_{\cal F}\|^{p-2}\Big\{
 p\,\<h_{\cal F},\, u\,(h_{\cal F}^2 + h_{\rm mix}^2) + {\rm Hess}^{\,\cal F}_{\,u}\> - n\,u\,\|h_{\cal F}\|^2\,{H} \Big\}{\rm d}V \\
 & = \int_M  u \Big\{ p\,(\nabla^{{\cal F}*})^2(\|h_{\cal F}\|^{p-2} h_{\cal F} )
 + p\,\|h_{\cal F}\|^{p-2}\big(\<h_{\cal F}, h_{\cal F}^2 + h_{\rm mix}^2\> - n\,\|h_{\cal F}\|^2\,{H}\big) \Big\}{\rm d}V .
\end{align}
From \eqref{E-I1} and \eqref{E-I2} the Euler-Lagrange equations \eqref{EVar-02} and \eqref{EVar1-02h} follow.
%
By \eqref{E-var-TS}$_2$ we get
\begin{align*}
 & \delta\,\sigma^{\cal F}_1 = \Delta_{\,\cal F}\,u + u\big(\sigma^{\cal F}_1 -2\,\sigma^{\cal F}_{2} +\|A_{\,\rm mix}\|^2,\\
 & \delta\,\sigma^{\cal F}_2 = \<T_{1}(A_{\cal F}), {\rm Hess}^{\cal F\sharp}_u\> + u\big(\sigma^{\cal F}_1\sigma^{\cal F}_{2}
 - 3\,\sigma^{\cal F}_{3} +\<T_{1}(A_{\cal F}),\, A_{\,\rm mix}^2\>\big).
\end{align*}
Using $Q^{\cal F}_2=\frac1{s^2(s-1)}\,((s-1)(\sigma^{\cal F}_1)^2-2\,s\,\sigma^{\cal F}_2)$ and \eqref{E-var-TS}$_2$ for $r=1,2$, we get
\begin{align*}
  s^2(s-1)\delta\,Q^{\cal F}_2 & = 2(s-1)\sigma^{\cal F}_1\delta\,\sigma^{\cal F}_1 - 2\,s\,\delta\,\sigma^{\cal F}_2
 = 2(s-1)\sigma^{\cal F}_1 \big(\Delta_{\,\cal F}\,u + u(\sigma^{\cal F}_1 -2\,\sigma^{\cal F}_{2} +\|A_{\,\rm mix}\|^2)\big) \\
 & - 2\,s\big(\<T_{1}(A_{\cal F}), {\rm Hess}^{\cal F\sharp}_u\>
 + u\,(\sigma^{\cal F}_1\sigma^{\cal F}_{2}  - 3\,\sigma^{\cal F}_{3} +\<T_{1}(A_{\cal F}),\, A_{\,\rm mix}^2\>)\big).
\end{align*}
Hence
\begin{align}\label{E-I4}
\nonumber
  \delta\,W^{\rm conf}_{n,s,2}
  & = \frac n2\int_M (Q^{\cal F}_2)^{n/2-1}\Big\{\frac1{s^2(s-1)}\big[
  2(s-1)\sigma^{\cal F}_1\big(\Delta_{\,\cal F}\,u + u(\sigma^{\cal F}_1 - 2\,\sigma^{\cal F}_{2} + \|A_{\,\rm mix}\|^2)\big) \\
 &\quad - 2\,s\big(\<T_{1}(A_{\cal F}), {\rm Hess}^{\cal F\sharp}_u\>
 + u\,(\sigma^{\cal F}_1\sigma^{\cal F}_{2} - 3\,\sigma^{\cal F}_{3} + \<T_{1}(A_{\cal F}), A_{\,\rm mix}^2\>)\big)\big]
  - 2\,Q^{\cal F}_2 u H\Big\}\,{\rm d}V .
\end{align}
Using \eqref{E-divP-2} with $f = (Q^{\cal F}_2)^{\frac n2-1}\sigma^{\cal F}_1$ and \eqref{EV2-nabla} with $B = (Q^{\cal F}_2)^{\frac n2-1}T_{1}(A_{\cal F})$
in \eqref{E-I4}, we get~\eqref{EVar1-conf}.
\end{proof}

\begin{remark}\rm
(i)~For a hypersurface $M\subset\mathbb{R}^{n+1}$ equipped with a line field (i.e., $s=1$) of the normal curvature $\kappa$,
the functionals \eqref{E-FFsigma} and \eqref{E-func} coincide with $WF_{n,1}=\int_M F(\kappa)\,{\rm d}V$ and $W_{n,p,1}=\int_M \kappa^p\,{\rm d}V$.
For~$W_{n,2,1}$ and $J_{n,2,1}$, from \eqref{EVar-02} and \eqref{EVar1-02h} with $p=2$ and $s=1$, using $(\nabla^{{\cal F}*})^2 h_{\cal F}=\Delta_{\,\cal F}\,\kappa$, we get the following leaf-wise elliptic PDE:
 $\Delta_{\,\cal F}\,\kappa + (\kappa^2 - \|h_{\rm mix}\|^2 - \frac n2\,H\kappa )\,\kappa = 0$.

(ii)~The first variation of the functional $W^{\rm conf}_{n.s,r}$ and the Euler-Lagrange equation can be obtained
from \eqref{E-main1ab}$_1$ and \eqref{E-main1a}$_1$, similarly to the corresponding equations in \cite{Guo-07} for $W^{\rm conf}_{n.r}$.

(iii)~By \eqref{E-I1}--\eqref{E-I2} with $s=n$, the first variations of functionals \eqref{E-func-0} are given by
\begin{align*}
 & \delta\,W_{n,p} = \int_M {H}^{p-1}\Big\{\frac pn\,\big(\Delta\,u + u\,\|{h}\|^2 \big) - n u {H}^2 \Big\}{\rm d}V ,\\
 & \delta\,J_{n,p} = \int_M \|{h}\|^{p-2}\Big\{p\,\<{h}, \ u\,{h}^2 + {\rm Hess}_{\,u}\> - n\,u\,\|{h}\|^2 {H} \Big\}{\rm d}V .
\end{align*}
The corresponding Euler-Lagrange equations are well known:
\begin{align}
\label{EVar-01}
 & \Delta\,{H}^{p-1} + {H}^{p-1}\,\big(\|{h}\|^2 - \frac{n^2}p\,{H}^2\big) = 0 ,\\
\label{EVar1-01b}
  & (\nabla^*)^2(\|{h}\,\|^{p-2} h) + \|{h}\,\|^{p-2}\big(\<h, h^2\> - \frac np\,\|{h}\,\|^{2} H \big) = 0 ,
\end{align}
for example, \cite[Corollary~1]{GTT-2019}, where $n=2$ and $M^2\subset\mathbb{R}^{3}$.
For $p=n=2$, we can use the identity
 $\|h\|^2 - 2\,H^2 = \frac12\,(k_1-k_2)^2= 2\,(H^2-K)$,
where $k_1$ and $k_2$ are the principal curvatures, $H=(k_1+k_2)/2$ and $K=k_1 k_2$ is the gaussian curvature of a surface $M^2\subset\mathbb{R}^3$.
In this case, the Euler-Lagrange equation \eqref{EVar-01} reduces to \eqref{Eq-Will-2}.
Using the identity $\<h, h^2\> = 8\,H^3 - 6\,H K$, the Euler-Lagrange equation \eqref{EVar1-01b} for $p=n=2$ reads~as
 $(\nabla^*)^2 h + 4\,H(H^2 - K) = 0$.
\end{remark}

\subsection{The second variation}
\label{sec:02b}

The following statement generalizes Corollary~1 in \cite{GTT-2019} when $M^2\subset\mathbb{R}^{3}$.

\begin{theorem}\label{Th-02}
If \eqref{E-divP-0} is valid, then the Euler-Lagrange equation for
\eqref{E-FFsigma}$_1$ with $F=F(H_{\cal F})$ is
\begin{align}\label{E-EpsB}
 \Delta_{\,\cal F}\,F' + F'(\|h_{\cal F}\|^2 - \|h_{\rm mix}\|^2) - s\,n\,F H  = 0.
\end{align}
At a critical hypersurface satisfying \eqref{E-divP-0}, the second variation of \eqref{E-FFsigma}$_1$ with $F=F(H_{\cal F})$~is
\begin{align}\label{E-07-Fin}
\nonumber
 &\delta^2 WF_{n,s} = -\!\int_M \frac{n}s\,\Big\{F'\Delta_{\,\cal F}\,u  -u\,\Delta_{\,\cal F}\,F'
 \Big\}u H\,{\rm d}V
 + \!\int_M \Big\{\frac{F'}s\,\big(2\,u\<h_{\cal F}, {\rm Hess}^{\,\cal F}_u\> + s\,u\<\nabla^{\cal F}H_{\cal F}, \nabla^{\cal F} u\> \\
\nonumber
 & + 2\,h(\nabla^{\cal F} u, \nabla^{\cal F} u) - s\,H_{\cal F}\|\nabla^{\cal F} u\|^2 \big)
 + \frac{F''}{s^2}\,\Delta_{\,\cal F}\,u\big(\Delta_{\,\cal F}\,u + u\,(\|h_{\cal F}\|^2 - \|h_{\rm mix}\|^2) \big)\Big\}{\rm d}V \\
\nonumber
 & + \int_M u\Big\{\big(\frac{F''}{s^2}\,(\|h_{\cal F}\|^2 - \|h_{\rm mix}\|^2) -\frac n s\,H F'\big)
 \big(\Delta_{\,\cal F}\,u + u\,(\|h_{\cal F}\|^2 - \|h_{\rm mix}\|^2)\big)
 - F (\Delta\,u + u\|h\|^2) \\
 & + \frac{F'}s\,\big(2\,\<h_{\cal F}, \ u\,(h_{\cal F}^2 + h_{\rm mix}^2) + {\rm Hess}^{\,\cal F}_{\,u}\>
 - u\,\<h_{\cal F}+h_{\cal F^\perp}, h_{\rm mix}^2\> - 2\,\<h_{\rm mix}, {\rm Hess}^{\,\rm mix}_{\,u}\> \big)
 \Big\}{\rm d}V .
\end{align}
\end{theorem}

\begin{proof}
By \eqref{E-main1a}$_1$ with $F
=F(\sigma_1^{\cal F}/s)$,
using $\<{\rm id}_{\,TM}, h_{\rm mix}^2\>=\|h_{\rm mix}\|^2$ and $\<{\rm id}_{\,T\cal F}, {\rm Hess}^{\cal F}_u\>=\Delta_{\,\cal F}\,u$, we find
the first variation of the functional \eqref{E-FFsigma}$_1$
with $F=F(H_{\cal F})$:
\begin{align}\label{E-EpsA}
 \delta\,WF_{n,s} = \int_M \Big\{\frac{F'}s\,\Delta_{\,\cal F}\,u + \big(\frac{F'}s\,(\|h_{\cal F}\|^2 -\|h_{\rm mix}\|^2) - n\,F H \big)\,u\Big\}{\rm d}V =0.
\end{align}
If \eqref{E-divP-0} is valid, then using \eqref{E-EpsA} and \eqref{E-divP-2} we obtain
\eqref{E-EpsB}.
Our next aim is to calculate
\begin{align}\label{E-36-m}
\nonumber
 & \delta^2 WF_{n,s}
 = \delta\int_M \Big\{\frac{F'}s\,\Delta_{\,\cal F}\,u + \big(\frac{F'}s\,(\|h_{\cal F}\|^2 - \|h_{\rm mix}\|^2) - n\,F H \big)\,u
 \Big\}{\rm d}V \\
\nonumber
 & = -\int_M\Big\{\frac{F'}s\,\Delta_{\,\cal F}\,u + \big(\frac{F'}s\,(\|h_{\cal F}\|^2 -\|h_{\rm mix}\|^2) -n\,F H \big)\,u\Big\}n\,u H\,{\rm d}V \\
 & + \int_M\delta\Big(\frac{F'}s\,\Delta_{\,\cal F}\,u\Big){\rm d}V
 + \int_M \delta\,\Big\{\Big(\frac{F'}s\,(\|h_{\cal F}\|^2 - \|h_{\rm mix}\|^2) - n\,F H \Big)\,u\Big\}{\rm d}V .
\end{align}
For the first integral in the last line of \eqref{E-36-m}, using \eqref{Eq-02a}, \eqref{E21} and $\delta u=0$, we~get
\begin{align}\label{E-38-m}
\nonumber
 & \int_M\delta\Big(\frac{F'}s\,\Delta_{\,\cal F}\,u\Big){\rm d}V
 = \int_M \Big\{\frac{F'}s\,\big(2\,u\<h_{\cal F}, {\rm Hess}^{\,\cal F}_u\> + s\,u\<\nabla^{\cal F}H_{\cal F}, \nabla^{\cal F} u\> \\
 & + 2\,h(\nabla^{\,\cal F} u, \nabla^{\cal F} u)
 - s H_{\cal F}\<\nabla^{\cal F} u, \nabla^{\cal F} u\> \big)
 + \frac{F''}{s^2}\,\Delta_{\,\cal F}\,u\big(\Delta_{\,\cal F}\,u + u\,(\|h_{\cal F}\|^2 - \|h_{\rm mix}\|^2) \big) \Big\}{\rm d}V .
\end{align}
For the second integral in the last line of \eqref{E-36-m}, using \eqref{E22}, we get
\begin{align}\label{E-39-m}
\nonumber
 & \int_M \delta\,\Big\{\Big(\frac{F'}s\,(\|h_{\cal F}\|^2 - \|h_{\rm mix}\|^2) - n\,F H \Big)\,u\Big\}{\rm d}V \\
\nonumber
 & = \int_M u\Big\{\big(\frac{F''}{s^2}\,(\|h_{\cal F}\|^2 - \|h_{\rm mix}\|^2) -\frac{n F'} s\,H\big)
 \big(\Delta_{\,\cal F}\,u + u\,(\|h_{\cal F}\|^2 - \|h_{\rm mix}\|^2)\big) - F (\Delta\,u + u\|h\|^2) \\
 & + \frac{F'}s\,\big(2\,\<h_{\cal F}, \ u\,(h_{\cal F}^2 + h_{\rm mix}^2) + {\rm Hess}^{\,\cal F}_{\,u}\>
 - u\,\<h_{\cal F}+h_{\cal F^\perp}, h_{\rm mix}^2\> - 2\,\<h_{\rm mix}, {\rm Hess}^{\,\rm mix}_{\,u}\>
 \big)
  \Big\}{\rm d}V .
\end{align}
By \eqref{E-36-m}, \eqref{E-38-m} and \eqref{E-39-m}, noting that the first variation vanishes at a critical immersion, we get
\begin{align}\label{E-07}
\nonumber
 &\delta^2 WF_{n,s}= -\int_M \frac ns\,\Big\{F'\Delta_{\,\cal F}\,u +\big(F'(\|h_{\cal F}\|^2 -\|h_{\rm mix}\|^2) -s\,n\,F H \big)\,u\Big\} u H\,{\rm d}V\\
\nonumber
 & + \int_M \Big\{\frac{F'}s\,\big( 2\,u\<h_{\cal F}, {\rm Hess}^{\,\cal F}_u\>
  + s\,u\<\nabla^{\cal F}H_{\cal F}, \nabla^{\cal F} u\> + 2\,h(\nabla^{\cal F} u, \nabla^{\cal F} u) - s\,H_{\cal F}\|\nabla^{\cal F} u\|^2 \big) \\
\nonumber
 & + \frac1{s^2}\,F''\Delta_{\,\cal F}\,u\big(\Delta_{\,\cal F}\,u + u\,(\|h_{\cal F}\|^2 - \|h_{\rm mix}\|^2) \big)\Big\}\,{\rm d}V \\
\nonumber
 & + \int_M u\Big\{\big(\frac{F''}{s^2}\,(\|h_{\cal F}\|^2 - \|h_{\rm mix}\|^2) -\frac{n F'} s\,H \big)
 \big(\Delta_{\,\cal F}\,u + u\,(\|h_{\cal F}\|^2 - \|h_{\rm mix}\|^2)\big)
 - F (\Delta\,u + u\|h\|^2) \\
 & + \frac{F'}s\,\big(2\,\<h_{\cal F}, \ u\,(h_{\cal F}^2 + h_{\rm mix}^2) + {\rm Hess}^{\,\cal F}_{\,u}\>
 - u\,\<h_{\cal F}+h_{\cal F^\perp}, h_{\rm mix}^2\> - 2\,\<h_{\rm mix}, {\rm Hess}^{\,\rm mix}_{\,u}\> \big)
 \Big\}{\rm d}V .
\end{align}
From \eqref{E-07} and \eqref{E-EpsB}, at a critical hypersurface, we get \eqref{E-07-Fin}.
\end{proof}

Similarly, one can obtain the Euler-Lagrange equation for the functional \eqref{E-FFsigma}$_2$ with $F=F(\|h_{\cal F}\|)$, we do not present it here.
From Theorem~\ref{Th-02} with $F=H_{\cal F}^{p}$ we obtain the following.

\begin{corollary}\label{Cor-03}
At a critical hypersurface satisfying \eqref{E-divP-0}, the second variation of the action \eqref{E-func}$_1$ is
\begin{align}\label{E-07b}
\nonumber
 & \delta^2 W_{n,p,s} = -\int_M \frac{n p}s\Big\{H_{\cal F}^{p-1}\Delta_{\,\cal F}\,u - u\,\Delta_{\,\cal F}\,(H_{\cal F}^{p-1})\Big\} u H\,{\rm d}V
 +\int_M \frac ps\,H_{\cal F}^{p-2}\Big\{H_{\cal F}\big(2\,u\<h_{\cal F}, {\rm Hess}^{\cal F}_u\> \\
\nonumber
 & + s\,u\<\nabla^{\cal F}H_{\cal F}, \nabla^{\cal F} u\> {+} 2 h(\nabla^{\cal F} u, \nabla^{\cal F} u) {-} s H_{\cal F}\|\nabla^{\cal F} u\|^2 \big)
 {+}\frac{p{-}1}{s}\Delta_{\,\cal F} u\big(\Delta_{\,\cal F} u {+} u\,(\|h_{\cal F}\|^2 {-} \|h_{\rm mix}\|^2) \big)\Big\}{\rm d}V \\
\nonumber
 & + \!\int_M \!H_{\cal F}^{p-2} u\Big\{\frac ps\Big(\frac{p{-}1}{s}(\|h_{\cal F}\|^2 {-} \|h_{\rm mix}\|^2) {-} n\,H H_{\cal F}\Big)
 \big(\Delta_{\,\cal F}\,u {+} u\,(\|h_{\cal F}\|^2 {-} \|h_{\rm mix}\|^2)\big)
 {-} H_{\cal F}^{2}(\Delta\,u {+} u\|h\|^2) \\
 &\quad + \frac{p}s\,H_{\cal F}\big(2\,\<h_{\cal F}, u\,(h_{\cal F}^2 + h_{\rm mix}^2) + {\rm Hess}^{\,\cal F}_{\,u}\>
 - u\,\<h_{\cal F}+h_{\cal F^\perp}, h_{\rm mix}^2\> - 2\,\<h_{\rm mix}, {\rm Hess}^{\,\rm mix}_{\,u}\> \big)
 \Big\}{\rm d}V .
\end{align}
\end{corollary}


\begin{remark}\em
By \eqref{E-07} with $s=n$, the second variation of the action $WF_{n}=\int_M F(H)\,{\rm d}V$ is
\begin{align}\label{E-07-G}
\nonumber
 &\delta^2 WF_n = -\!\int_M \Big\{F'\,\Delta\,u {+} \big(F'\|h\|^2 {-} n^2\,F H \big)\,u\Big\} u\,H\,{\rm d}V
 {+} \!\int_M \Big\{\frac{F'}n\big(2\,u\<{h}, {\rm Hess}_u\> {+} n\,u\<\nabla\,{H}, \nabla u\> \\
\nonumber
 & + 2\,h(\nabla\,u, \nabla u) {-} n H\,\|\nabla u\|^2 \big) + \frac{F''}{n^2}\Delta\,u\big(\Delta\,u + u\,\|h\|^2\big)\Big\}{\rm d}V \\
 & + \int_M u\Big\{\big(\frac{F''}{n^2}\,\|h\|^2 - H F' - F \big)\big(\Delta\,u + u\,\|h\|^2 \big)
  +\frac{2\,F'}n\,\<h,\,u{h}^2 + {\rm Hess}_{\,u}\> \Big\}{\rm d}V .
\end{align}
This is compatible with a special case of Eq.~(7) in \cite{GTT-2019} for $n=2$.
As a special case of \eqref{E-07-G}, the second variation of the functional \eqref{E-func-0}$_1$ with $n=2$ has the following form
compatible with \cite[Eq.~(45)]{GTT-2019}:
\begin{align*}
 & \delta^2 W_{2,p} = \!\int_M H^{p-2}\Big\{\frac{p(p-1)}4\,(\Delta\,u)^2 + p\,H \big(h(\nabla\,u, \nabla\,u) + 2\,u\<h, {\rm Hess}_u\>
  + u \<\nabla\,H, \nabla u\> - H\|\nabla u\|^2\big) \\
 & + \big((2p^2{-}4p-1)H^2 {-} p(p{-}1) K \big)u\Delta\,u + (4p(p-1)H^{4} {-} 2(p-1)(2p+1) K H^2 {+} p(p{-}1) K^2)u^2 \Big\}{\rm d}V.
\end{align*}
\end{remark}

\section{Applications}
\label{sec:04}

We consider critical hypersurfaces equipped with two-dimensional foliations (i.e., $s=2$) in Sect.~\ref{sec:04a},
and discuss critical hypersurfaces of revolution
and their stability in Sect.~\ref{sec:04b}.

\subsection{Hypersurfaces with two-dimensional foliations}
\label{sec:04a}


For $s=2$, it~is natural to present the functionals \eqref{E-FFsigma} in the following form:
\begin{equation}\label{E-func-KH}
 WF_{n,2} = \int_M F({H}_{\cal F}, K_{\cal F})\,{\rm d}V .
\end{equation}
where
${K}_{\cal F}=\det A_{\cal F} = \sigma^{\cal F}_2$
is the Gaussian curvature of the leaves.
For~$n=s=2$, \eqref{E-func-KH} reduces to the functional $WF_{2}=\int_M F(H, K)\,{\rm d}V$ seen in \cite{GTT-2019}.
The~following equalities are true:
\[
 \|h_{\cal F}\|^2 = k_{{\cal F},1}^2 + k_{{\cal F},2}^2 = 4\,{H}_{\cal F}^2 - 2\,K_{\cal F},\quad
 \<h_{\cal F}, \ h_{\cal F}^2\,\> = k_{{\cal F},1}^3 + k_{{\cal F},2}^3 = 8\,{H}_{\cal F}^3 - 6\,{H}_{\cal F} K_{\cal F}.
\]
From \eqref{Eq-02a} and \eqref{Eq-02h} with $s=2$, we obtain the following evolution equations:
\begin{align}
\label{Eq-02N2}
 & \delta\,(2\,{H}_{\cal F}) = \Delta_{\,\cal F}\,u + u\,( 4\,{H}_{\cal F}^2 - 2\,K_{\cal F} - \|h_{\rm mix}\|^2) ,\\
 \label{Eq-02hN2}
 & \delta\,\|h_{\cal F}\|^2 = 4\,u\,(4\,{H}_{\cal F}^3 - 3\,{H}_{\cal F} K_{\cal F})
 + 2\,\<h_{\cal F}, \ u\,h_{\rm mix}^2 + {\rm Hess}^{\,\cal F}_{\,u}\> .
\end{align}
Using \eqref{Eq-02N2} and \eqref{Eq-02hN2} in $\delta\,K_{\cal F} = 2\,{H}_{\cal F}\,\delta\,(2\,{H}_{\cal F}) - (1/2) \delta\,(\|h_{\cal F}\|^2)$, we get the evolution equation
\begin{align}\label{Eq1-02K}
 \delta\,K_{\cal F}
 = 2\,{H}_{\cal F}\,\Delta_{\,\cal F}\,u - \<h_{\cal F}, \ u\,h_{\rm mix}^2 + {\rm Hess}^{_{\cal F}}_{\,u}\>
 + 2\,u\,{H}_{\cal F}\,(K_{\cal F} - \|h_{\rm mix}\|^2) .
\end{align}
For $n=2$, \eqref{Eq-02N2}--\eqref{Eq1-02K} reduce to the equations in \cite[Lemma~1]{GTT-2019}.

The next statement for $(M^n,{\cal F}^2)$ immersed in $\mathbb{R}^{n+1}$ generalizes Theorem~1 in \cite{GTT-2019} with $n=2$.

\begin{theorem}\label{Th-03}
If \eqref{E-divP-0} is valid, then the Euler-Lagrange equation for the action \eqref{E-func-KH} with $s=2$~is
\begin{align}\label{E-Eps2}
\nonumber
 & \Delta_{\,\cal F}\Big(\frac12\, F'_H + 2\,{H}_{\cal F} F'_K\Big) - (\nabla^{{\cal F}*})^2(F'_K h_{\cal F})
 + F'_H \big(2\,H_{\cal F}^2 - K_{\cal F} - \frac12\,\|h_{\rm mix}\|^2 \big) \\
 &\qquad + F'_K \big(2\,{H}_{\cal F}\,(K_{\cal F} - \|h_{\rm mix}\|^2) - \<h_{\cal F}, \ h_{\rm mix}^2\,\> \big) -n\,F H = 0 ,
\end{align}
where $F'_H, F'_K$ denote partial derivatives of $F(H_{\cal F},K_{\cal F})$ with respect to $H_{\cal F}$ and $K_{\cal F}$.
At a critical hypersurface foliated by surfaces $(s=2)$ and satisfying \eqref{E-divP-0}, the second variation of the functional
\eqref{E-func-KH} with $F=F(H_{\cal F})$ is given by
\begin{align}\label{E-07-2H}
\nonumber
 &\delta^2 WF_{n,2} = -\!\int_M \frac n2\,\Big\{F'\Delta_{\,\cal F}\,u - u\,\Delta_{\,\cal F}\,F' \Big\}u H\,{\rm d}V
 + \!\int_M \Big\{\frac{F'}2\big(2\,u\<h_{\cal F}, {\rm Hess}^{\,\cal F}_u\> {+} 2\,u\<\nabla^{\cal F}H_{\cal F}, \nabla^{\cal F} u\> \\
\nonumber
 &\quad + 2\,h(\nabla^{\cal F} u, \nabla^{\cal F} u) - 2\,H_{\cal F}\|\nabla^{\cal F} u\|^2 \big)
 + \frac{F''}{4}\,\Delta_{\,\cal F}\,u\big(\Delta_{\,\cal F}\,u + u\,(4\,H_{\cal F}^2 - 2\,K_{\cal F} - \|h_{\rm mix}\|^2) \big)\Big\}{\rm d}V \\
\nonumber
 &\quad  + \int_M u\Big\{\Big(\frac{F''}{4}\,(4\,H_{\cal F}^2 - 2\,K_{\cal F} - \|h_{\rm mix}\|^2) -\frac{n F'}2\,H \Big)
 \big(\Delta_{\,\cal F}\,u + u\,(4\,H_{\cal F}^2 - 2\,K_{\cal F} - \|h_{\rm mix}\|^2)\big) \\
\nonumber
 &\quad + \frac{F'}2\big(4\,u\,{H}_{\cal F}(4{H}_{\cal F}^2 - 3\,{K}_{\cal F})
 + 2\,u\,\<h_{\cal F},  h_{\rm mix}^2\> +2\,\<h_{\cal F},  {\rm Hess}^{\,\cal F}_{\,u}\>
 - F (\Delta\,u {+} u\|h\|^2) \\
 & \quad- u\<h_{\cal F} + h_{\cal F^\perp}, h_{\rm mix}^2\> - 2\<h_{\rm mix}, {\rm Hess}^{\,\rm mix}_{\,u}\> \big)\Big\}{\rm d}V .
\end{align}
\end{theorem}

\begin{proof}
Using \eqref{Eq-02N2} and \eqref{Eq1-02K} in
$\delta\,F(H_{\cal F}, K_{\cal F}) = \frac12 F'_H\cdot\delta\,(2\,H_{\cal F}) + F'_K\,\delta\,K_{\cal F}$,
and applying \eqref{Eq-01d}, we calculate the first variation of the functional \eqref{E-func-KH} with $s=2$:
\begin{align}\label{E-Eps}
\nonumber
 & \delta\,WF_{n,2} = \int_M \delta\big(F(H_{\cal F}, K_{\cal F})\,{\rm d}V\big)
 = \int_M \Big\{\frac12\, F'_H \big(\Delta_{\,\cal F}\,u + u\,( 4\,H_{\cal F}^2 - 2\,K_{\cal F} - \|h_{\rm mix}\|^2) \big) \\
\nonumber
 & + F'_K \big(2\,{H}_{\cal F}\,\Delta_{\,\cal F}\,u - \<h_{\cal F}, \ u\,h_{\rm mix}^2 + {\rm Hess}^{_{\cal F}}_{\,u}\>
 + 2\,u\,{H}_{\cal F}\,(K_{\cal F} - \|h_{\rm mix}\|^2)\big) -n\,u\,H F \Big\}{\rm d}V \\
\nonumber
 & = \int_M \Big\{\big(\frac12\, F'_H + 2\,{H}_{\cal F} F'_K\big)\,\Delta_{\,\cal F}\,u - F'_K \<h_{\cal F}, \ {\rm Hess}^{_{\cal F}}_{\,u}\>
 + \big(F'_H \big(2\,H_{\cal F}^2 - K_{\cal F} - \frac12\,\|h_{\rm mix}\|^2\big) \\
 &\quad + F'_K\big(2\,{H}_{\cal F}\,(K_{\cal F} - \|h_{\rm mix}\|^2) - \<h_{\cal F}, \ h_{\rm mix}^2\,\>\big) - n\,F H \big)\,u\Big\}{\rm d}V .
\end{align}
From \eqref{E-Eps}, using \eqref{EV2-nabla}, we get \eqref{E-Eps2}.
From Theorem~\ref{Th-02} with $s=2$ we get \eqref{E-07-2H}.
\end{proof}

\begin{remark}\rm
Let $F=F({H}_{\cal F})$,
 then from \eqref{E-Eps2} we obtain
\begin{align*}
 \Delta_{\,\cal F}\big(F'_H \big) + F'_H \big(4\,H_{\cal F}^2 - 2\,K_{\cal F} - \|h_{\rm mix}\|^2 \big) -2\,n F H = 0 .
\end{align*}
From this with $F={H}_{\cal F}^2$,
or, from \eqref{EVar-02},
we get the Euler-Lagrange equation for $W_{n,2,2}$,
see \eqref{E-func}$_1$:
\begin{align*}
 \Delta_{\,\cal F}\,{H}_{\cal F} + \big(\,4{H}_{\cal F}^2 - 2\,{K}_{\cal F} - \|h_{\rm mix}\|^2 - n\,H {H}_{\cal F}\big) {H}_{\cal F} = 0 .
\end{align*}
\end{remark}

The following particular case of \eqref{E-07-2H}, or, Corollary~\ref{Cor-03} with $s=2$, is true.

\begin{corollary}
At a critical hypersurface satisfying \eqref{E-divP-0}, the second variation of $W_{n,p,2}$~is
\begin{align*}
\nonumber
 & \delta^2 W_{n,p,2} = -\int_M \frac{n p}2\Big\{H_{\cal F}^{p-1}\Delta_{\,\cal F}\,u - u\,\Delta_{\,\cal F}\,(H_{\cal F}^{p-1})\Big\} u H\,{\rm d}V
 +\int_M \frac p2\,H_{\cal F}^{p-2}\Big\{2\,H_{\cal F}\big(u\<h_{\cal F}, {\rm Hess}^{\cal F}_u\>  \\
\nonumber
 & + u\<\nabla^{\cal F}H_{\cal F}, \nabla^{\cal F} u\> {+} h(\nabla^{\cal F} u, \nabla^{\cal F} u) {-} H_{\cal F}\|\nabla^{\cal F} u\|^2 \big)
 {+}\frac{p{-}1}{2}\Delta_{\,\cal F} u\big(\Delta_{\,\cal F} u {+} u(4\,H_{\cal F}^2 {-} 2\,K_{\cal F} {-} \|h_{\rm mix}\|^2) \big)\Big\}{\rm d}V \\
\nonumber
 & + \!\int_M \!H_{\cal F}^{p-2} u\Big\{ \frac p2\Big(\frac{p-1}{2}\,(4\,H_{\cal F}^2 {-} 2 K_{\cal F} {-} \|h_{\rm mix}\|^2) - n\,H H_{\cal F}\Big)
 \big(\Delta_{\,\cal F}\,u + u\,(4\,H_{\cal F}^2 {-} 2 K_{\cal F} {-} \|h_{\rm mix}\|^2)\big) \\
\nonumber
 &\quad - H_{\cal F}^{2}(\Delta\,u + u\|h\|^2) + \frac{p}2\,H_{\cal F}\Big(4\,u\,(4{H}_{\cal F}^3 - 3\,{H}_{\cal F} {K}_{\cal F}^2)
 + 2\,\<h_{\cal F}, \ u\,h_{\rm mix}^2+{\rm Hess}^{\,\cal F}_{\,u}\> \\
 &\quad - u\,\<h_{\cal F}+h_{\cal F^\perp}, h_{\rm mix}^2\> - 2\,\<h_{\rm mix}, {\rm Hess}^{\,\rm mix}_{\,u}\> \Big) \Big\}{\rm d}V .
\end{align*}
\end{corollary}

The following consequence of Theorem~\ref{Th-03} was proven for $n=2$ in \cite{GTT-2019}.

\begin{corollary}
The Euler-Lagrange equation for the functional $WF_{2}$ with $F=F(H,K)$ is given by
\begin{align*}
 \Delta\Big(\frac12\,F'_H + 2\,H F'_K\Big) - (\nabla^{*})^2(F'_K h) + \big(2\,H^2 - K\big) F'_H + 2\,H K F'_K - n\,H F = 0 .
\end{align*}
\end{corollary}

\subsection{Hypersurfaces of revolution}
\label{sec:04b}

Hypersurfaces of revolution in Euclidean space $\mathbb{R}^{n+1}$ are naturally foliated into $(n-1)$-spheres (pa\-rallels)
and equipped with rotationally symmetric metrics $g=d\rho^2+\rho^2 ds^2_{n-1}$ -- a special case of a warped product metric,
see \cite[Sect.~4.2.3]{petersen}.
Such a hypersurface can be represented as a graph $x_{n+1}=f(\rho)$, where
the function $f\in C^2$ is monotonic, $\rho=\sqrt{x_1^2+\ldots+x_{n}^2}$
and $(x_i)$ are Cartesian coordinates in~$\mathbb{R}^{n+1}$.
We~obtain the parametrization ${\bf r} = {\bf r}(\phi_1,\ldots,\phi_{n-1}; f(\rho))$,
of $M$, where $(\phi_1,\ldots,\phi_{n-1};\rho)$ are cylindrical coordinates in $\mathbb{R}^{n+1}$.
The~principal curvatures of $M$ (functions of $\rho)$ are
$k_1=\ldots=k_{n-1}=\frac{f'}{\rho\,(1+(f')^2)^{1/2}}\le\frac1\rho$ for parallels and
$k_n=\frac{f''}{(1+(f')^2)^{3/2}}$ for profile curves -- geodesics on $M$.
If the profile curve is a straight line ($f''\equiv0$), then $k_n\equiv0$ and $M$ is a cone, or a cylinder, or a hyperplane.
To exclude these cases, we will assume $f''\not=0$.
We~get
\begin{align}\label{E-hHF}
\nonumber
 & n\,H=(n-1)\,k_1+k_n,\quad H_{\cal F}=k_1,\quad \|h_{\cal F}\|^2=(n-1)k_1^2, \quad \|h\|^2 = (n-1)k_1^2+k_n^2, \\
 & \<h, {h}^2\> = (n{-}1)\,k_1^3+k_n^3,\quad \<h_{\cal F}, h_{\cal F}^2\> = (n{-}1)k_1^3, \quad h_{\rm mix}=h_{\rm mix}^2=0 .
\end{align}
Recall that $\lambda_j=j(j+n-2)$ correspond to the solutions with multiplicities $N_j=C^n_{n+j}=\frac{(n+j)!}{n!\,j!}$
of the eigenvalue problem $\Delta\,u + \lambda\,u = 0$ on a unit $(n-1)$-sphere.
Any constant function on the round sphere spans the space of $\lambda_0$-eigenfunctions of the Laplacian.
Let $\perp$ denote the orthogonality of functions with respect to the $L^2$ inner product.
The sphere $S^2(\rho)$ of radius $\rho$
in $\mathbb{R}^3$ is not a local minimum of $W_{2,p}$ under volume-preserving deformations for $p>2$.
For $p\ge1$, $S^2(\rho)$ is a local minimum of $W_{2,p}$ under volume-preserving,
nonconstant deformations $u$ provided $u\in \{v: \Delta\,v = (2/\rho^2)v\}^\perp$, see Propositions~2 and 3 in \cite{GTT-2019}.
According to \eqref{E-EpsB}, a hypersurface of revolution
in $\mathbb{R}^{n+1}$ foliated by $(n-1)$-spheres-parallels $\{L_\rho\}$ is a critical point of the action
\eqref{E-FFsigma}$_1$ with $F=F(H_{\cal F})$ if and only if
\begin{align*}
 (F'/F)\,((n-1)k_1^2+k_n^2) - s\,n\,((n-1)k_1+k_n) = 0.
\end{align*}
In this case, $k_n$ and $k_1$ are functionally related; hence, $M$ is a Weingarten hypersurface.

The following theorem studies the stability of hypersurfaces of revolution critical for \eqref{E-func}.

\begin{theorem}
A hypersurface of revolution $M: x_{n+1}=f(\rho)\ (f''\not=0)$ in $\mathbb{R}^{n+1}$ foliated by $(n-1)$-spheres-parallels $\{L_\rho\}$ is a critical point of the action $W_{n,p,n-1}$ or $J_{n,p,n-1}$, see \eqref{E-func}, if and only~if
\begin{align}\label{E-ODE-mu1}
 f(\rho)=\int\frac{\sqrt {C_1 \rho^{2p-2n+2} - \rho^{4p-4n+4}}}{C_1-\rho^{2p-2n+2}}\,{\rm d}\rho+C_2,\quad C_1,C_2\in\mathbb{R}.
\end{align}
A critical hypersurface is not a local minimum of
$W_{n,p,n-1}$ for $p>n\ge2$ with respect to general variations, but it is a local minimum
for variations $u=u(\phi_1,\ldots,\phi_{n-1})$ satisfying $u|_{L_\rho}\perp \ker\Delta_{\,\cal F}$.
\end{theorem}

\begin{proof}
1. Let $M$ be critical for the action $W_{n,p,n-1}$
under general deformations.
Since all principal curvatures $k_i$ are
 constant on parallels,
from \eqref{EVar-02} and \eqref{EVar1-02h} we get, respectively,
\begin{align}\label{EE-01}
 p\,\|h_{\cal F}\|^2 - n(n-1)\,H {H}_{\cal F}=0,\quad
 p\,\<h_{\cal F}, \ h_{\cal F}^2\> - n\,\|h_{\cal F}\|^2\,H = 0 .
\end{align}
Using \eqref{E-hHF} in \eqref{EE-01}, yields $k_n=(p-n+1)\,k_1\ne0$, which is the differential equation for $f=f(\rho)$,
\begin{align}\label{E-ODE-mu}
 \rho\,f''=(p-n+1)\,f'(1+(f')^2) .
\end{align}
The solution of \eqref{E-ODE-mu} is given by \eqref{E-ODE-mu1}.

2. Let $u=u(\phi_1,\ldots,\phi_{n})$ be the eigenfunction of $\Delta$ on $S^{n-1}(1)$
with the eigenvalue $\lambda_j$, then
 $\Delta_{\,\cal F}\,u + \lambda_j\,\rho^{-2}\,u = 0$.
Since our hypersurface of revolution $(M,g)$ is a warped product, its volume form is decomposed as ${\rm d}V={\rm d}V_\rho\cdot d\,\rho$,
see~\cite[Note~7.1.1.1]{ber}.
For any function $a(\rho)$ we have, see~\eqref{E-divP-1},
\begin{align*}
 \int_M a\|\nabla^{\,\cal F} u\|^2 {\rm d}V
 = \!\int_{\rho_1}^{\rho_2} \!a\big(\int_{L_\rho}\|\nabla^{\,\cal F} u\|^2 {\rm d}V_\rho\big) d\,\rho
 = -\!\int_{\rho_1}^{\rho_2} \!a\big(\int_{L_\rho} u\,\Delta_{\,\cal F}\,u\ {\rm d}V_\rho\big) d\,\rho
 = -\!\int_M a\,u\,\Delta_{\,\cal F}\,u\,{\rm d}V .
\end{align*}
Using \eqref{E-07b} with $s=n-1$, Example~\ref{Ex-01}, the equalities $\nabla u=\nabla^{\,\cal F} u$ and
\begin{align*}
 \<\nabla\,k_1, \nabla u\> = 0,\
 h(\nabla^{\,\cal F} u, \nabla u)= k_1\|\nabla^{\cal F} u\|^2,\
 \<h, {\rm Hess}_{\,u}\> = \<h_{\cal F}, {\rm Hess}^{\,\cal F}_u\> = (n-1)\,{H}_{\cal F}\Delta_{\,\cal F}\,u ,
\end{align*}
we find the second variation of $W_{n,p,n-1}$:
\begin{align}\label{E-07c2}
\nonumber
 & \delta^2 W_{n,p,n-1}
 = \frac1{(n-1)^2}\int_M k_1^{p-2}\big\{p(p-1)(\Delta_{\,\cal F}\,u)^2 \\
 & +(n-1)(5\,np-n -9\,p +1)k_1^2\,u\,\Delta_{\,\cal F}\,u
  -(n-1)^2(p-n)(p-n+1)k_1^4\,u^2\big\}{\rm d}V .
\end{align}
If the variation $u=u(\phi_1,\ldots,\phi_{n-1})$ satisfies $\Delta_{\,\cal F}\,u=0$, then by \eqref{E-07c2} we~get
\begin{align*}
 \delta^2 W_{n,p,n-1} = -\int_M (p-n)(p-n+1)k_1^{p+2} u^2\,{\rm d}V,
\end{align*}
which is negative for $p>n$; hence, our critical hypersurface is not a local minimum of $W_{n,p,n-1}$.

Let
$u$ satisfy $u|_{L_\rho}\perp \ker\Delta_{\,\cal F}$.
 Using the inequalities
 $\int_{S^{n-1}(1)} u\,\Delta\,u\,{\rm d}V \ge \int_{S^{n-1}(1)} \lambda_1 u^2\,{\rm d}V$
 and
 $\int_{S^{n-1}(1)} (\Delta\,u)^2\,{\rm d}V \ge \int_{S^{n-1}(1)} \lambda_1 u^2\,{\rm d}V$,
see, for example, \cite{GTT-2019}, we get
\begin{align*}
 \int_{L_\rho} u\,\Delta_{\,\cal F}\,u\,{\rm d}V_{L_\rho} \ge \int_{L_\rho} \lambda_1\rho^{-4} u^2\,{\rm d}V_{L_\rho},\quad
 \int_{L_\rho} (\Delta_{\,\cal F}\,u)^2\,{\rm d}V_{L_\rho} \ge \int_{L_\rho} \lambda_1\rho^{-4} u^2\,{\rm d}V_{L_\rho}.
\end{align*}
By these estimates, \eqref{E-07c2} and the inequality $\rho^{-4}\ge k_1^2$, we find
\begin{align*}
 \nonumber
 & \delta^2 W_{n,p,n-1} \ge \frac1{(n-1)^2}\int_M \Big\{p(p-1) (n-1)^2 +(n-1)^2(5\,np -n -9\,p+1) \\
 &-(n-1)^2(p-n)(p-n+1)\Big\}k_1^{p+2} u^2{\rm d}V = \int_M \big\{n(p-n) + p(6\,n-11) +1 \big\}k_1^{p+2} u^2{\rm d}V .
\end{align*}
Hence, $\delta^2 W_{n,p,n-1}>0$ for all $p\ge n\ge2$.
\end{proof}

\begin{example}\label{Ex-01}\rm
{(i)} Let $M^3: x_{4}=f(\rho)\ (f''\not=0)$, $\rho=\sqrt{x_1^2+x_2^2+x_{3}^2}$ be a
hypersurface of revolution in $\mathbb{R}^{4}$
foliated by parallels (2-spheres). 
We get $H=\frac13\,(2\,k_1+k_3)$ and $H_{\cal F}=k_1$.
Let~$M^3$ be critical for the functional $W_{3,p,2}$ or $J_{3,p,2}$ with $p\ge3$, see \eqref{E-func}, under general deformations.
From \eqref{EE-01} we get $k_3=(p-2)k_1\ne0$.
Solution of \eqref{E-ODE-mu} is
 $f(\rho)=\!\int\!\frac{\sqrt {C_1 \rho^{2 p+4} - \rho^{4 p}}}{\rho^{2 p} - C_1\,\rho^{4}}\,{\rm d}\rho+C_2$, where $C_1,C_2\in\mathbb{R}$.

{(ii)}
Let $M^2: x_{3}=f(\rho)\ (f''\not=0)$, $\rho=\sqrt{x_1^2+x_2^2}$ be a surface of revolution in $\mathbb{R}^{3}$ foliated by parallels (circles).
The principal curvatures
are $k_1=\frac{f'}{\rho(1+(f')^2)^{1/2}}$ for parallels
and $k_2=\frac{f''}{(1+(f')^2)^{3/2}}$ for profile curves.
Let $M^2$ be critical for the action $W_{2,p,1}$ or $J_{2,p,1}$ with $p\ge2$
under general deformations.
Then $k_2=(p-1)k_1\ne0$;
hence,
the equality $H^2/K =\frac{p^2}{4\,(p-1)}$ is true.
Solution of \eqref{E-ODE-mu} is
 $f(\rho)=\!\int\frac{\sqrt {C_1 \rho^{2 p+2} - \rho^{4 p}}}{\rho^{2 p}-C_1\,\rho^{2}}\,{\rm d}\rho+C_2$, where $C_1,C_2\in\mathbb{R}$,
it is illustrated on Fig.~\ref{f-01} for $p=2,3,\ldots,8$.
\begin{figure}[ht]
\begin{center}
\includegraphics[scale=0.4]{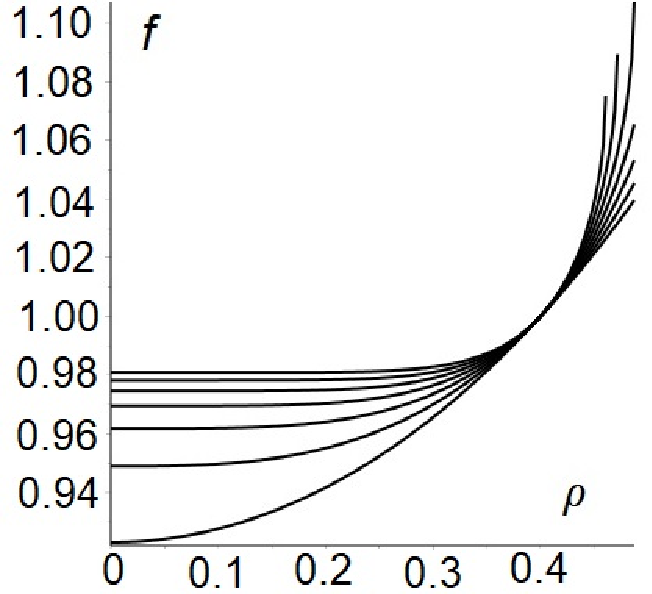}
\caption{\small Graphs of $f(\rho)$ for $f(\frac25)=1$, $f'(\frac25)=\frac25$, $n=2$ and $p=2,3,\ldots,8$.\label{f-01}}
\end{center}
\end{figure}
\end{example}

\section{Appendix}
\label{sec:07}

\begin{proof} (of Lemma~\ref{Le-01}).
Using $\delta\,{\bf r}_i = u_i\,{\bf N} + u\,{\bf N}_i$, we calculate
\begin{align*}
 \<\delta\,{\bf r}_i, {\bf r}_j\> = \<u_i\,{\bf N} + u\,{\bf N}_i, {\bf r}_j\>
 = \,u\,\<{\bf N}_i, {\bf r}_j\> = - u\,\<{\bf N}, {\bf r}_{ij}\> = - u\,h_{ij}.
\end{align*}
Thus, since the symmetry $h_{ij}=h_{ji}$ we get the equality \eqref{Eq-01a}:
 $\delta\,g_{ij} = \<\delta\,{\bf r}_i, {\bf r}_j\> + \<{\bf r}_i, \delta\,{\bf r}_j\> = - 2\,u\,h_{ij}$.
From $g^{il}g_{lj}=\delta^i_j$ it follows that
 $(\delta g^{il})g_{lj} = - g^{il} (\delta g_{lj}) = 2\,u\,g^{il} h_{lj}$;
hence, \eqref{Eq-01b} is true.

 We will compute the variation of $h$.
Using $\<{\bf N}, {\bf N}_i \>=0$, we find
\begin{align*}
 \<{\bf N},\ \delta\,{\bf r}_{ij}\> = \<{\bf N},\ (u\,{\bf N})_{ij}\>
  = u_{ij} - u\,\<{\bf N}_i,\ {\bf N}_{j}\>
  = u_{ij} - u\,\<{h}^l_i\,{\bf r}_l,\ {h}^k_j\,{\bf r}_k\>
   = u_{ij} - u\,{h}^l_i\,h_{jl}.
\end{align*}
Note that $\delta\,{\bf N}=c^i{\bf r}_i$ for some functions $c^i$. Using $\<{\bf N}, {\bf r}_j\>=0$, we get
\begin{align*}
 g_{ij}\,c^i = \<\delta\,{\bf N},\ {\bf r}_j\> = -\<{\bf N},\ \delta\,{\bf r}_j\> = -\<{\bf N},\ u_j\,{\bf N} + u\,{\bf N}_j\> = -u_j.
\end{align*}
It follows that $c^i=-g^{ij}u_j$ and $\delta\,{\bf N}=-g^{ij}u_j{\bf r}_i=-u^i{\bf r}_i$.
Using the Gauss equation for a hypersurface in $\mathbb{R}^{n+1}$, we get
at $x$:
 $\<\delta\,{\bf N},\ {\bf r}_{ij}\>=\<-u^i{\bf r}_i,\ h_{jl}{\bf N} +\Gamma^k_{jl}\,{\bf r}_k\>=0$.
Thus, \eqref{E-h-beta} is true:
\begin{align*}
 \delta {h}_{ij} = \delta \<{\bf r}_{ij},\ {\bf N}\> = \<\delta\,{\bf N},\ {\bf r}_{ij}\> + \<{\bf N},\ \delta\,{\bf r}_{ij}\>
 = u_{ij} - u {h}^l_i\,h_{jl}.
\end{align*}
Calculating the variation of the mean curvature, we get \eqref{Eq-01e}:
\begin{align*}
 & \delta\,(n H) = \delta\,(g^{ij} h_{ij}) = 2\,u {h}^{ij}\,h_{ij} + g^{ij}\big(u_{ij} - u {h}^l_{i}\,h_{jl}\big) \\
 & =  u\,h_{ij}\,{h}^{ij} + g^{ij} u_{ij} = \Delta u + u\,\|{h}\|^2.
\end{align*}
The formula $\delta({\rm d}V) = \frac{1}{2}\,(\tr_g \delta\,g)\,{\rm d}V$
for variation of ${\rm d}V$ is valid for any variation $\delta g$ of a metric.
for example, \cite{R2021}.
Applying
\eqref{Eq-01a} to the above
gives \eqref{Eq-01d}.
Next, we calculate the variation of $\|{h}\|^2$:
\begin{align*}
 &\quad \delta\,\|{h}\|^2
 = \delta\,(g^{ik} g^{jl}h_{kl} h_{ij}) \\
 & = 2\,u \big({h}^{ik} g^{jl} + {h}^{jl} g^{ik}\big)h_{kl} h_{ij}
 + g^{ik} g^{jl}\big((u_{kl} - u{h}^q_{k}h_{lq})h_{ij} + (u_{ij} - u{h}^q_{i}h_{jq})h_{kl}\big)\\
 & =  2\,(uh_{ij} {h}^{ik} {h}^j_{k} + u_{kl}\,g^{ik}g^{jl})
   = 2\,u\,\<{h},\ {h}^2\> + 2\,\<{h}, {\rm Hess}_{\,u}\>,
\end{align*}
that proves \eqref{Eq-01h}.
\end{proof}

\begin{proof} (of Lemma~\ref{Lem-03})
First, using \eqref{Eq-01a} and \eqref{E-h-beta}, we get for $1\le i,j\le s$ and $1\le q\le n$,
\begin{align*}
 \delta\,(s\,{H}_{\cal F}) = \delta\,(g^{ij} h_{ij}) = 2\,u {h}^{ij}\,h_{ij} + g^{ij}\big(u_{ij} - u {h}^q_{i}\,h_{jq}\big)
 = \Delta_{\,\cal F}\,u + u\,(\|h_{\cal F}\|^2 - \|h_{\rm mix}\|^2) ,
\end{align*}
that proves \eqref{Eq-02a}. Also for $1\le i,j,k,l\le s$ and $1\le q\le n$ we obtain \eqref{Eq-02h}:
\begin{align*}
 & \delta\,\|h_{\cal F}\|^2 = \delta (g^{ik} g^{jl} h_{kl} h_{ij}) \\
 & = 2\,u \big({h}^{ik} g^{jl} + {h}^{jl} g^{ik}\big)h_{kl} h_{ij}
 + g^{ik} g^{jl}\big((u_{kl} - u{h}^q_{k}h_{lq})h_{ij} + (u_{ij} - u{h}^q_{i}h_{jq})h_{kl}\big)\\
 & = 4\,u\,\<h_{\cal F}, h_{\cal F}^2 \> + 2\,\<h_{\cal F}, {\rm Hess}^{\,\cal F}_{\,u}\>
   -2\,u\,\<h_{\cal F}, h_{\cal F}^2 \> - 2\,u\,\<h_{\cal F}, h_{\rm mix}^2\> .
\end{align*}
For $1\le i,j,k,l\le s$, $s<\alpha,\beta,\gamma\le n$ and $1\le q\le n$ we obtain \eqref{E22}:
\begin{align*}
 & \delta\,\|h_{\rm mix}\|^2 = \delta (g^{ij} g^{\alpha\beta} h_{i\alpha} h_{j\beta}) \\
 & = 2\,u \big({h}^{ij} g^{\alpha\beta} + {g}^{ij} h^{\alpha\beta}\big)h_{i\alpha} h_{j\beta}
 + g^{ij} g^{\alpha\beta}\big(u_{i\alpha}h_{j\beta} + u_{j\beta}h_{i\alpha}\big)
 -u\,g^{ij} g^{\alpha\beta}\big(h^l_i h_{\alpha l} h_{j\beta} + h^l_j h_{\alpha i} h_{l\beta} \big)\\
 & = \<h_{\cal F}+h_{\cal F^\perp}, u\,h_{\rm mix}^2\> + 2\,\<h_{\rm mix}, {\rm Hess}^{\,\rm mix}_{\,u}\> .
\end{align*}

The proof of \eqref{E21} is similar to the proof of \cite[Eq.~(19)]{GTT-2019}: instead of $M^2$ we consider $s$-dimensional leaves of $\cal F$.
The variation of the Christoffel symbols is the following tensor, e.g., \cite{GTT-2019}:
\begin{align}\label{E-dGamma}
 \delta\,\Gamma^k_{ij} = - u\,g^{kl}\big(h_{jl,i} + h_{il,j} - h_{ij,l} \big) - g^{kl}\big(u_i h_{jl} + u_j h_{il} - u_l h_{ij}\big).
\end{align}
For the Laplacian $\Delta_{\,\cal F}\,f = g^{ij}\,(f_{ij} - \Gamma^k_{ij}\,f_k)$ with $1\le i,j,k\le s$ it follows that
\begin{align}\label{Eq-Lapl}
 \delta\,(\Delta_{\,\cal F}\,f) = \delta\,(g^{ij} f_{ij}) - \delta\,(g^{ij}\Gamma^k_{ij} f_k).
\end{align}
For the first term we get (for $1\le i,j\le s$)
\begin{align}\label{E25}
 \delta\,(g^{ij}f_{ij}) = 2\,u\,h^{ij} f_{ij} + g^{ij}\dot f_{ij} = 2\,u\,\<h_{\cal F},\,{\rm Hess}^{\,\cal F}_f\> + \Delta_{\,\cal F}\,\dot f.
\end{align}
For the second term, using \eqref{E-dGamma} and $\Gamma^k_{ij}=0$ at $x$, we get for $1\le i,j,k\le s$ and $1\le q\le n$,
\begin{align}\label{E26}
\nonumber
 & \delta\,(g^{ij}\,\Gamma^k_{ij} f_k) = g^{ij}\delta\,(\Gamma^k_{ij})f_k
 = - g^{ij} g^{kq}\big\{u\,(h_{jq,i} + h_{iq,j} - h_{ij,q})f_k - (u_i h_{jq} + u_j h_{iq} - u_q h_{ij})f_k\big\} \\
 & = -2\,u\,g^{ij}\,g^{kq} h_{jq,i} f_k + u\,g^{ij}\,g^{kq} h_{ij,q} f_k - 2\,g^{ij}\,g^{kq}u_i h_{jq}f_k + g^{ij}\,g^{kq}\,u_q h_{ij}f_k .
\end{align}
Using the Codazzi-Minardi equation
 $\nabla_k\,h_{ij} -\nabla_j\,h_{ik} = 0$,
e.g., \cite{petersen},
we get for $1\le i,j,k,l\le s$,
\begin{align*}
 g^{ij}\,g^{kl}(\nabla_j\,h_{il}) f_k = g^{kl}(g^{ij}\nabla_l\,h_{ij}) f_k
 = s\,(\nabla^k H_{\cal F}) f_k
 = s\,\<\nabla^{\cal F} H_{\cal F},\, \nabla^{\cal F} f\> .
\end{align*}
Thus, using normal coordinates and
 $-2\,u\,g^{ij} g^{kl}h_{jl,i}\,f_k + u\,g^{ij} g^{kl} h_{ij,l} f_k = -u\,g^{kl}(g^{ij} h_{ij,l}) f_k$,
we get
 $s\,(H_{\cal F})_l
 = g^{ij}\,\nabla_l\,h_{ij}$
for $1\le i,j,l\le s$.
Therefore, \eqref{E26} becomes
\begin{align}\label{E31}
\nonumber
 & \delta\,(g^{ij}\,\Gamma^k_{ij}f_k) = -u\,g^{kl}(g^{ij} h_{ij,l}) f_k - 2\,g^{ij} g^{kl} u_i h_{jl}f_k + g^{kl} u_l(g^{ij} h_{ij}) f_k \\
 & = -s\,u\,\<\nabla^{\cal F}H_{\cal F}, \nabla^{\cal F}f\> - 2\,h(\nabla^{\cal F} u, \nabla^{\cal F} f)
 + s\,H_{\cal F}\<\nabla^{\cal F} u, \nabla^{\cal F} f\> .
\end{align}
Applying \eqref{E25} and \eqref{E31} to \eqref{Eq-Lapl}
completes the proof of \eqref{E21}.
%
\end{proof}


\begin{proof} (of Lemma~\ref{Lem-02})
We have $\Delta_{\,\cal F}\,f_2 = {\rm div}_{\cal F}(\nabla^{\,\cal F}f_2)$.
One can show that
 ${\rm div}_{\cal F}\,(PX) = {\rm div}\,(PX) - ({\rm div}\,P)(PX)$
 for all $X\in\mathfrak{X}_M$.
Hence, using
$\nabla^{\,\cal F}f = P\nabla f$ and \eqref{E-divP-0},
we get
\begin{align*}
 & f_1\Delta_{\,\cal F}\,f_2 = f_1\,{\rm div}_{\,\cal F}(\nabla^{\,\cal F} f_2) = f_1\big\{{\rm div}(P \nabla f_2) - ({\rm div}\,P)(P\nabla f_2) \big\}\\
 & = {\rm div}\,(f_1\,P\nabla f_2) - \<P\nabla f_1, P\nabla f_2\>
  = {\rm div}\,(f_1\,P\nabla f_2) - \<\nabla^{\,\cal F}f_1, \nabla^{\,\cal F}f_2\> .
\end{align*}
Using the Divergence Theorem,
gives
\begin{align}\label{E-divP-1}
 \int_M f_1(\Delta_{\,\cal F}\,f_2)\,{\rm d}V = -\int_M \big\<\nabla^{\,\cal F}f_1, \nabla^{\,\cal F}f_2\big\>\,{\rm d}V.
\end{align}
By this,
the
formula \eqref{E-divP-2} is true.
 Next, using \eqref{E-divP-0} we will prove
\begin{align}\label{E-cond-PP-int}
 \int_M \< \varphi_1, \nabla^{\cal F} \varphi_2\>\,{\rm d}V = \int_M  \< \nabla^{\cal F*} \varphi_1, \varphi_2\>\,{\rm d}V
\end{align}
for any compactly supported $(s,t)$-tensor $\varphi_1$ and $(s,t+1)$-tensor $\varphi_2$.
Define a compactly supported 1-form $\omega$ by
 $\omega(Y) = \<\iota_{\,Y}\,\varphi_2,\,\varphi_1\>$ for $Y\in{\mathcal X}_{M}$.
Take an orthonormal frame $(e_{i})$ such that $\nabla_Y e_{i}=0$ for all $Y\in T_xM$ and some $x\in M$.
To simplify calculations, assume that $s=t=1$, then at $x\in M$,
\begin{align*}
 & -\nabla^{*{\cal F}}\omega = \sum\nolimits_{j} (\nabla^{\cal F}_{e_j}\,\omega)(e_j)
 = \sum\nolimits_{i,j} \<{P} e_{j}, e_{i}\>\,(\nabla_{e_{i}}\,\omega)(e_{j}) \\
 & = \sum\nolimits_{i,j,c} \<{P} e_{j}, e_{i}\>
 \big(\<\nabla_{e_{i}} \varphi_2(e_{j},e_{c}), \varphi_1(e_{c})\> + \<\varphi_2(e_{j},e_{c}), \nabla_{e_{i}} \varphi_1(e_{c})\>\big) \\
 & = \!\sum\nolimits_{i,j,c} \big[\<\,\<{P} e_{j}, e_{i}\>\nabla_{e_{i}}\varphi_2(e_{j},e_{c}), \varphi_1(e_{c})\>
 {+} \<\varphi_2(e_{j},e_{c}),\<{P} e_{j}, e_{i}\>\nabla_{e_{i}}\varphi_1(e_{c}) \> \big] \\
 & = \sum\nolimits_{j,c}\big[ \<\nabla^{\cal F}_{e_{j}}\varphi_2(e_{j},e_{c}), \varphi_1(e_{c})\>
 +\<\varphi_2(e_{j},e_{c}), \nabla^{\cal F}_{e_{j}}\varphi_1(e_{c})\>\big]
 = \<\varphi_2, \nabla^{\cal F} \varphi_1\> -\<\nabla^{*{\cal F}}\varphi_2, \varphi_1\>.
\end{align*}
The $\nabla^{{\cal F}*}$ is related to the ${\cal F}$-{divergence} of a~vector field $\omega^\sharp$ by
 ${\rm div}_{\cal F}\,\omega^\sharp  = -\nabla^{{\cal F}*}\,\omega$.
By the above
and
$\int_M ({\rm div}_{\cal F}\ \omega^\sharp)\,{\rm d}V=\int_M {\rm div}(P\omega^\sharp)\,{\rm d}V=0$,
we obtain \eqref{E-cond-PP-int}.
Applying this twice, we get \eqref{EV2-nabla}.
\end{proof}

\end{document}